\documentclass[11pt]{article}

\usepackage{pifont}  

\usepackage{amsmath}
\usepackage{amsfonts}
\usepackage{color}
\usepackage{float}   
\usepackage{amssymb}
\usepackage{graphicx,bm}
\usepackage{enumerate}
\usepackage{amsthm,amscd}
\usepackage[all]{xy}

 \allowdisplaybreaks

 \usepackage[colorlinks, linkcolor=red, anchorcolor=green, citecolor=blue]{hyperref}

\newtheorem{theorem}{Theorem}[section]
\newtheorem{corollary}[theorem]{Corollary}

\newtheorem{conjecture}[theorem]{Conjecture}

\newtheorem{lemma}[theorem]{Lemma}

\newtheorem{problem}[theorem]{Problem}

\newcommand{\blue}{\color{blue}}

\begin{document}

\title{Spectral supersaturation: Triangles and bowties}

\author{Yongtao Li$^{1}$ 
\quad  
Lihua Feng$^{1,}$\footnote{Corresponding authors. This paper was published on European Journal of Combinatorics. \\ 
E-mail addresses:  \url{ytli0921@hnu.edu.cn} (Y. Li), 
\url{fenglh@163.com} (L. Feng), 
\url{ypeng1@hnu.edu.cn} (Y. Peng)}  
\quad 
Yuejian Peng$^{2,*}$  \\
{\small $^{1}$School of Mathematics and Statistics,  Central South University,  Changsha, China} \\ 
 {\small $^{2}$School of Mathematics, Hunan University, Changsha, China}  
 }

\date{\today}
 
\maketitle

\vspace{-0.8cm}

\begin{abstract}
A classical result of Erd\H{o}s and Rademacher (1955) demonstrates a fundamental supersaturation phenomenon in extremal combinatorics: every graph on $n$ vertices with more than $\lfloor n^2/4\rfloor$ edges contains at least $\lfloor n/2 \rfloor$ triangles.  
Let $\lambda (G)$ be the spectral radius of the adjacency matrix of a graph $G$. 
Recently, Ning and Zhai (2023) proved that 
every $n$-vertex graph $G$ with $\lambda (G) \ge 
\sqrt{\lfloor n^2/4\rfloor}$ contains at least $\lfloor n/2\rfloor -1$ triangles, 
unless $G$ is a balanced complete bipartite graph $K_{\lceil \frac{n}{2} \rceil, \lfloor \frac{n}{2} \rfloor}$. 
The aim of this paper is two-fold. 
Using a different approach which we term 
the supersaturation-stability method, 
we prove a stability variant of the Ning--Zhai result by showing that such a graph $G$ contains at least $n-3$ triangles if no vertex lies in all triangles of $G$. This bound is the best possible and it could also be viewed as a spectral analogue of a theorem of  Xiao and Katona (2021), which guarantees $n-2$ triangles under the assumption that $e(G)> \lfloor n^2/4\rfloor$ and  no vertex is in all triangles of $G$. 

The second part concerns with the spectral supersaturation for the bowtie, which consists of two triangles sharing a vertex. 
Erd\H{o}s, F\"{u}redi, Gould and Gunderson (1995) proved 
that every $n$-vertex graph with more than $\lfloor n^2/4\rfloor +1$ edges contains a bowtie.   
The spectral supersaturation problem has not been investigated for non-color-critical substructures in graphs with given order. 
We give the first such result by counting bowties.   
Let $K_{\lceil \frac{n}{2} \rceil, \lfloor \frac{n}{2} \rfloor}^{+2}$ 
be the graph obtained from $K_{\lceil \frac{n}{2} \rceil, \lfloor \frac{n}{2} \rfloor}$ by embedding two disjoint edges in the vertex part of size $\lceil \frac{n}{2} \rceil$.  We show that if $n\ge 8.8 \times 10^6$ and 
$\lambda (G)\ge \lambda (K_{\lceil \frac{n}{2} \rceil, \lfloor \frac{n}{2} \rfloor}^{+2})$, then $G$ has at least $\lfloor \frac{n}{2} \rfloor$  bowties,  
and $K_{\lceil \frac{n}{2} \rceil, \lfloor \frac{n}{2} \rfloor}^{+2}$ is the unique spectral extremal graph.  
This gives a spectral correspondence of a result of Kang, Makai and Pikhurko (2020). 
The method developed in our paper could be helpful in establishing the spectral results for counting other substructures, even for non-color-critical graphs.  
 \end{abstract}




\section{Introduction} 
Extremal combinatorics is one of the essential branches of 
discrete mathematics and it has experienced an impressive growth 
in the last few decades. 
Generally speaking, it deals with the problem of determining or estimating the maximum or minimum possible size of a combinatorial object that satisfies certain restrictions.  
Such types of problems are often related to other mathematical areas, such as  theoretical computer science, coding theory, 
discrete geometry, 
ergodic theory and harmonic analysis. 
One of the most significant problems in extremal combinatorics  is the famous Tur\'{a}n type problem, 
which can be roughly formulated as follows: 
given a `forbidden' substructure, 
what is the maximum number of edges in a graph 
that does not contain the specific substructure.  
The first study of the Tur\'{a}n type problem dates back perhaps to a classical result of Mantel, 
which asserts that each graph on $n$ vertices 
with more than $\lfloor n^2/4\rfloor $ edges must contain a triangle. 
The bound is tight as witnessed by 
a balanced complete bipartite graph $T_{n,2}$. 
The Tur\'{a}n type problem could be viewed as 
 a starting point of extremal graph 
theory and it has inspired extensive research in the past hundred years. 
A large number of results have  
been proved for a range of different substructures using a wide variety of different methods; 
see, e.g., \cite{Bollobas78}. 

\medskip 
A result of Rademacher (unpublished, see Erd\H{o}s \cite{Erd1955,Erdos1964}) shows that 
any $n$-vertex graph with more than $\lfloor n^2/4\rfloor $ edges has at least $\lfloor n/2\rfloor$ triangles. 
 This bound can be achieved 
 by adding an edge to the larger vertex part of $T_{n,2}$. 
 This phenomenon is known as `supersaturation' in the literature.  
Subsequently, 
Lov\'{a}sz and Simonovits \cite{LS1975,LS1983}  proved that 
if $1\le t <{n}/{2}$ is an integer and $G$ is an $n$-vertex graph with
$e(G)\ge \lfloor {n^2}/{4} \rfloor + t$,  
then $G$ contains at least $t \lfloor {n}/{2}\rfloor $
triangles.   
An analogous question for other classes of graphs 
has been studied extensively in extremal graph theory.  
We refer the readers to 
the works of Mubayi \cite{Mub2010}, and 
Pikhurko and Yilma \cite{PY2017} 
for counting color-critical graphs, 
such as cliques, odd cycles, etc. 
The stability results on Erd\H{o}s--Rademacher's problem 
can be found in \cite{XK2021,LM2022-Erd-Rad,BC2023}.

\medskip 
Another interesting research direction 
studies which other substructures 
must appear in a graph with more than $\lfloor n^2/4\rfloor $ edges.  
Apart from the number of triangles, 
 a similar extremal graph problem concerns the minimum 
 number of edges that 
are contained in a triangle. 
For instance, every $n$-vertex graph with more than 
 $ \lfloor {n^2}/{4} \rfloor$ edges has 
 at least $2\lfloor n/2\rfloor +1$ edges contained in triangles;  see, e.g.,  
\cite{EFR1992,FM2017,GL2018}. 
Other examples are results on finding many cliques 
that share one or more vertices and edges, 
such as,  
every graph on $n$ vertices with more than 
 $ \lfloor {n^2}/{4} \rfloor$ edges contains a book of size 
 $n/6$, that is,  $n/6$ 
 triangles sharing a common edge; see \cite{BN2005,BN2011,CFS2020-Siam}. 
 Moreover,  each $n$-vertex graph with more than $ \lfloor {n^2}/{4} \rfloor +1$ edges contains a bowtie, i.e., 
 two triangles sharing a vertex; 
see \cite{Erdos95, Chen03,KMP2020}.

\subsection{Counting triangles via the spectral radius}

In the past few years, 
an extremely popular trend has been the expansion of 
spectral extremal graph theory, in which the structural properties of a graph are studied by means of eigenvalues of the associated matrices.  
In this paper, we focus mainly on 
the adjacency matrix $A(G)=[a_{i,j}]_{n \times n}$, 
where $a_{i,j}=1$ if $v_i$ and $v_j$ are adjacent in $G$, 
and $a_{i,j}=0$ otherwise. 
The spectral radius $\lambda (G)$ of $G$
 is defined as the maximum modulus 
of eigenvalues of $A(G)$. 

\medskip 
The spectral Tur\'{a}n problem 
studies the maximum spectral radius of 
a graph that forbids specific subgraphs.  
For example, a spectral Mantel theorem due to Nikiforov \cite{Niki2007laa2} states that 
a graph $G$ of order $n$ with $\lambda (G)> \lambda (T_{n,2})$ contains a triangle. 
This spectral version extends Mantel's theorem  
by invoking the fact  $\lambda (G) \ge 2m/n$. 
The spectral Tur\'{a}n problems have attracted considerable attention in the last twenty years, continuing to be a fascinating and long-standing topic in the field; see, e.g., 
\cite{Wil1986,Niki2002cpc,Niki2007laa2,LP2022second,ZhangST2024} for graphs without cliques, 
\cite{LNW2021,ZS2022dm,LFP2023-solution,LLH2022} for non-bipartite graphs without triangles,  
\cite{Niki2009cpc} for a spectral Erd\H{o}s--Stone--Bollob\'{a}s theorem,  
\cite{Niki2009jgt} for a spectral Erd\H{o}s--Simonovits stability theorem, 
\cite{BG2009,Niki2010laa} for a spectral  K\H{o}vari--S\'{o}s--Tur\'{a}n theorem, 
\cite{CDT2023} for a spectral Erd\H{o}s--S\'{o}s theorem, 
\cite{LLP2024-AAM} for a spectral Erd\H{o}s--Rademacher theorem, 
\cite{LFP2024-triangular} for a spectral Erd\H{o}s--Faudree--Rousseau theorem, 
\cite{Niki2021,ZL2022jgt} for books, 
\cite{CDT21,ZHL2021} for wheels, 
\cite{Wang2022} for a spectral extremal result on a class of graphs, 
\cite{CLZ2019-GC} for disjoint paths, 
\cite{NWK2023} for disjoint cliques, 
 \cite{LL2024} for 
disjoint color-critical graphs, 
\cite{TT2017,LN2021outplanar} for planar and outerplanar graphs, and 
\cite{FLS2024, ZL2024} for disjoint cycles in planar graphs.

\medskip 
There are a large number of 
articles involving the results on counting substructures in the literature.  
Nevertheless, it is extremely \textbf{rare} to establish the counting results 
in terms of the spectral radius of a graph. 
In 2007, Bollob\'{a}s and Nikiforov \cite{BN2007jctb}  showed that every $n$-vertex graph $G$ contains 
at least $  (\frac{\lambda (G)}{n} -1 + \frac{1}{r}) 
\frac{r(r-1)}{r+1}(\frac{n}{r})^{r+1} $ copies of $K_{r+1}$. 
This result reveals the spectral supersaturation for complete graphs. 
Namely, for any $r\ge 2$ and $\varepsilon>0$, if  $G$ is a graph with 
$\lambda (G) \ge (1-\frac{1}{r} + \varepsilon)n$, then $G$ has at least $\delta (r,\varepsilon)\cdot n^{r+1}$ copies of $K_{r+1}$, 
where $\delta (r,\varepsilon):=\frac{r-1}{(r+1)r^r} \varepsilon$ 
is linearly dependent on $\varepsilon$. 
And beyond that, 
a recent result due to
 Ning and Zhai \cite{NZ2021} 
established an exact spectral supersaturation 
for triangles.

\begin{theorem}[Ning--Zhai, 2023] \label{thmNZ2021}
If $G$ is an $n$-vertex graph  with 
\[   \lambda (G) \ge \lambda (T_{n,2}), \]  
then  $G$ contains at least 
$ \lfloor \frac{n}{2} \rfloor -1 $ triangles, 
unless $G=T_{n,2}$.  
\end{theorem}

In the sequel, we shall make an effort
 to the study on spectral counting problems. 
 For a graph $G$ with $\lambda (G)\ge 
 \lambda (T_{n,2})$, 
 we establish structural results and improve the bound on the number of triangles.   
In particular, we give a spectral stability by showing that if 
a graph $G$ satisfies $\lambda (G) \ge \lambda(T_{n,2})$ 
and $G$ has few triangles, 
then it is close to a complete bipartite graph. 
With this structural result, we present a stability variant on Theorem \ref{thmNZ2021}.

\subsection{Extremal problems for friendship graphs} 

\label{sub-sec1-2}

A graph $G$ is $F$-free if it does not contain a subgraph isomorphic to $F$.  
 By convention, we write $\mathrm{ex}(n,F)$ for 
the maximum number of edges in an $F$-free graph of order $n$. 
An $F$-free graph 
with the maximum number of edges 
is called an extremal graph for $F$. 
We denote by $\mathrm{Ex}(n,F)$ 
the set of all $n$-vertex extremal graphs for $F$. 

\medskip 
Let $F_k$ be the graph of order $2k+1$ 
consisting of $k$ triangles
that intersect in exactly one common vertex.
This graph is known as the friendship graph
because it is the unique extremal graph of the
well-known Friendship Theorem\footnote{It states that if $G$ is a finite graph in which any two vertices have
precisely one common neighbor, then there is a vertex which is adjacent to
all other vertices and  $G=F_k$ for some $k$; see \cite[page 307]{AZ2014}}. 
In 1995, Erd\H{o}s, F\"{u}redi,
Gould and Gunderson \cite{Erdos95} proved that  
for every $k \geq 1$ and $n\geq 50k^2$, we have 
\[ \mathrm{ex}(n, F_k)= \left\lfloor \frac {n^2}{4}\right \rfloor+ \left\{
\begin{array}{ll}
k^2-k \quad~~  \mbox{if $k$ is odd; } \\
k^2-\frac32 k \quad \mbox{if $k$ is even}.
\end{array}
\right. \] 
Moreover, the extremal graphs in $\mathrm{Ex}(n,F_k)$  
 are determined in \cite{Erdos95}. For odd $k$,
the extremal graphs are constructed from $T_{n,2}$
by embedding two vertex-disjoint copies of $K_k$ in one side. 
For even $k$, the extremal graphs
are obtained by taking $T_{n,2}$ and embedding
any of the graphs with $2k-1$ vertices, $k^2- {3}k/2$ edges and maximum degree $k-1$ in one side.

\medskip   
In 2020, Cioab\u{a}, Feng,
Tait and Zhang \cite{CFTZ20} showed that 
for every fixed $k\ge 2$ and sufficiently large $n$, 
if $G$ is an $F_k$-free graph on $n$ vertices 
with the maximum spectral radius, then
$  G \in  \mathrm{Ex}(n,F_k)$. 
This work introduces a valuable approach to characterize the spectral extremal graphs when we forbid non-bipartite graphs.  
Furthermore, Zhai, Liu and Xue \cite{ZLX2022} provided a complete 
characterization by determining the unique spectral extremal graph for $F_k$, 
i.e., they found out the unique embedded graph among 
all graphs  of order $2k-1$ with $k^2-3k/2$ edges 
and maximum degree $k-1$. 
In \cite{CFTZ20}, the celebrated triangle removal lemma\footnote{It states that for all $\varepsilon >0$, there exists $\delta >0$ such that 
any graph on $n$ vertices with at most $\delta n^3$ triangles can be made triangle-free by removing at most $\varepsilon n^2$ edge. 
The best bound on $\delta^{-1}$ 
is a tower function with base $2$ and logarithmic height in $\varepsilon^{-1}$; see \cite{CF2013}. So, it leads to a worse bound on the order of the  graph.} was applied so that 
it must require the order $n$ to be sufficiently large.  
Recently, 
Li, Feng and Peng \cite{LFP2024-triangular}  presented  
a  new simple method that avoids the use of triangle removal lemma and 
shows the main result \cite{CFTZ20} holding for every $n\ge (21k)^4$.

 \begin{figure}[htbp]
\centering
\includegraphics[scale=0.18]{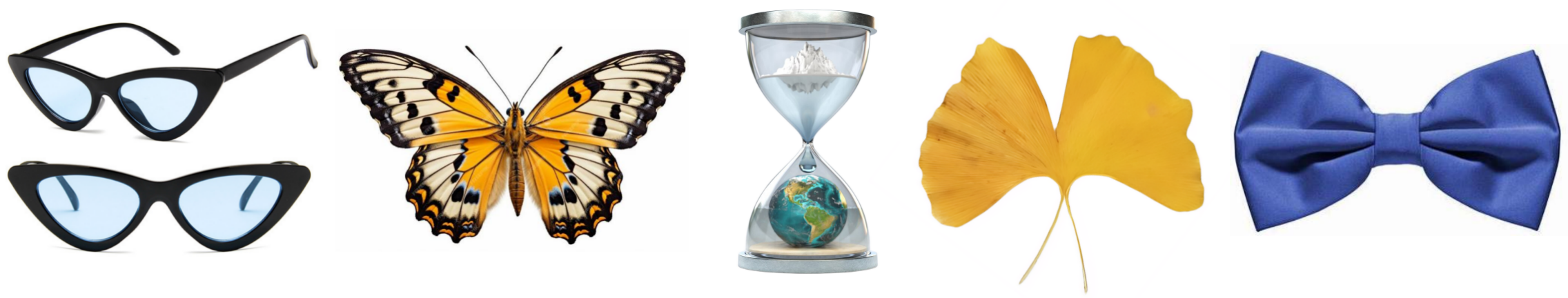} 
\vspace{-4mm}
\caption{Some pictures whose shapes resemble $F_2$.}
\label{fig-bowtie}
\end{figure}

\medskip 
 In this paper, we investigate 
 the spectral supersaturation for the graph $F_2$, which consists of two triangles 
merged at a vertex. 
The graph $F_2$ is also known as the {\it bowtie}.  
The bowtie-free graphs are  extremely 
important in several areas of graph theory,
such as in the Ramsey problem \cite{HN2018} and the 
supersaturation problem \cite{KMP2020}. 
In fact, the Tur\'{a}n number of bowtie
was also obtained by
Erd\H{o}s, F\"{u}redi, Gould and Gunderson \cite{Erdos95} for every $n\ge 5$.

\begin{theorem}[Erd\H{o}s et al., 1995] \label{EFGG-F2}
For $n\ge 5$,  we have 
\[  \mathrm{ex}(n,F_2) 
= \left\lfloor \frac{n^2}{4} \right\rfloor +1. \]
The extremal graphs are obtained from $T_{n,2}$ by adding 
an edge to one part.  
\end{theorem}

In 2020, Kang, Makai and Pikhurko \cite{KMP2020} 
proved that the only way to construct a graph of order $n$ with more than $\lfloor n^2/4\rfloor +1$ edges and containing as few bowties as possible is to add additional edges to the Tur\'{a}n graph $T_{n,2}$.   

\begin{theorem}[Kang--Makai--Pikhurko, 2020] \label{thm-KMP}
There exists $\delta >0$ such that 
if $1\le q\le \delta n^2$ and 
$H$ is a graph on $n\ge 1/ \delta$ vertices with 
$ \lfloor n^2/4 \rfloor+ q$ edges, 
and $H$ has the minimum number of copies of the bowtie $F_2$, 
then $H$ contains $T_{n,2}$ as a subgraph. 
\end{theorem}

Another direction concerns the counting results on other substructures,  
such as the triangular edges\footnote{An edge is called {\it triangular} if it is contained in a triangle.}, books and bowties, which 
can be guaranteed by the adjacency spectral radius.  
The spectral results for the triangular edges and the books were recently studied in \cite{LFP2024-triangular} and \cite{ZL2022jgt}, respectively.  
Motivated by these results, the second aim in our study is to 
count the number of bowties and 
investigate a spectral correspondence of Theorem \ref{thm-KMP}.

\section{Main results}

We are interested in finding 
the maximum number of substructures under certain spectral conditions. 
These types of questions arise in a wide variety of contexts. 
In this study, we will present two main results on 
the spectral stability and supersaturation 
by counting the number of triangles as well as bowties.

\subsection{A stability of Ning--Zhai's result}
\label{sec1.1}

The extremal graph in Erd\H{o}s--Rademacher's theorem 
is obtained from $T_{n,2}$ by adding an edge to the part of size $\lceil n/2\rceil$. 
In 2021, 
Xiao and Katona \cite{XK2021} 
presented a stability variant 
 by showing that 
if $G$ has more than $ \lfloor {n^2}/{4}  \rfloor  $ edges 
and there is no vertex contained in all triangles,  
then $G$ has at least $n-2$ triangles. 
The {\it triangle covering number} $\tau_3(G)$ is the minimum number of vertices in a set that contains at least one vertex of each triangle in $G$.  
We define $\tau_3(G)=0$ when $G$ contains no triangles.   
In particular,  $\tau_3(G)=1$ if and only if all triangles of $G$ share a common vertex.

\begin{theorem}[Xiao--Katona, 2021] \label{thm-XK}
If $G$ is an $n$-vertex graph with  $e(G)> \lfloor {n^2}/{4}  \rfloor $
 and 
  the triangle covering number  $\tau_3(G) \ge 2$, 
  then $G$ contains at least $n-2$ triangles.  
\end{theorem}

The first main result of this paper 
gives a spectral version of Theorem \ref{thm-XK} as follows. 
Moreover, our result could be viewed as a stability of 
Ning--Zhai's result in Theorem \ref{thmNZ2021}. 
  

\begin{theorem} \label{thm-n-2}
Let $G$ be a graph on $n\ge 113$ vertices with 
\[ \lambda (G) \ge  
 \lambda (T_{n,2}). \]  
 If the triangle covering number  $\tau_3(G)\ge 2$, 
then  $G$ has at least $n-3$ triangles. 
\end{theorem}

The bound in Theorem \ref{thm-n-2} 
is the best possible.  We can find the extremal graphs with exactly $n-3$ triangles. 
 For even $n$, we can delete an edge  
 from $K_{\frac{n}{2} +1, \frac{n}{2} -1}^{+2}$.  
 For odd $n$, we can take $K_{\frac{n+3}{2}, \frac{n-3}{2}}^{+2}$, or delete two edges from 
$K_{\frac{n+1}{2},\frac{n-1}{2}}^{+2}$ to destroy two triangles. 
This yields the following graphs in Figure \ref{fig-six}. 
For simplicity, we omit the tedious calculations.

 \begin{figure}[htbp]
\centering
\includegraphics[scale=0.79]{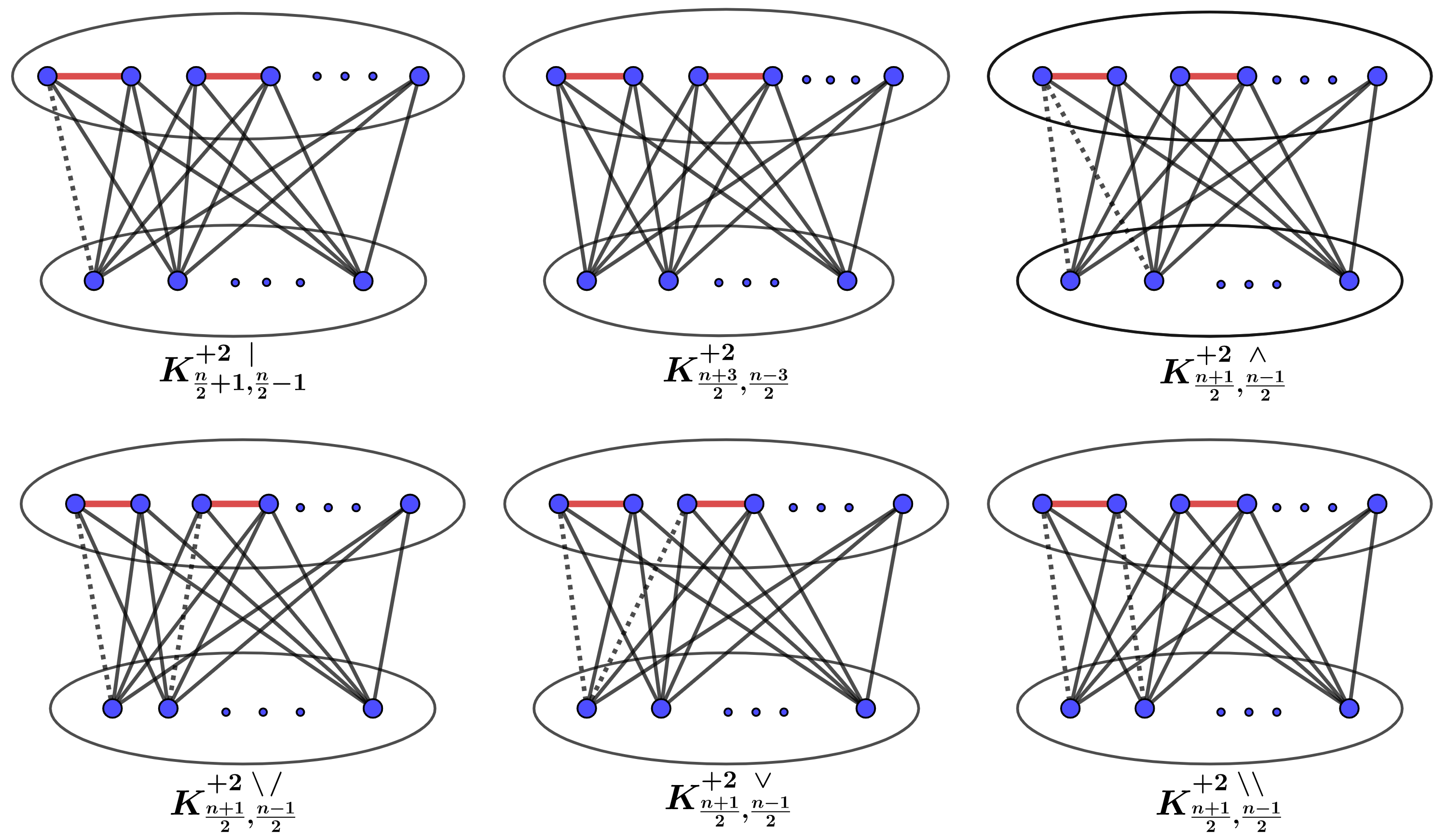} 
\caption{The extremal graphs in Theorem \ref{thm-n-2}.}
\label{fig-six}
\end{figure}

\subsection{Spectral supersaturation for the bowtie}

\label{sec1.2}

The spectral extremal problem for the 
friendship graph $F_k$ was investigated in \cite{CFTZ20,ZLX2022}  
for sufficiently large $n$. 
As mentioned before, there are many interesting extremal problems 
 involving the bowtie $F_2$. 
For example, Li, Lu and Peng \cite{2022LLP} studied 
the problem on $F_2$ and got rid of the condition on $n$ being sufficiently large. 
The authors \cite{2022LLP} presented a spectral version of Theorem \ref{EFGG-F2} for every $n\ge 7$. 
We denote by $K_{ \lfloor \frac{n}{2} \rfloor, \lceil \frac{n}{2} \rceil}^+$ the graph obtained from
 $K_{ \lfloor \frac{n}{2} \rfloor, \lceil \frac{n}{2} \rceil}$
by embedding an edge into the vertex part of size $\lfloor \frac{n}{2} \rfloor$; see Figure \ref{fig-thm6-7}.

\begin{theorem}[Li--Lu--Peng, 2023] \label{thm-n-F2}
If $G$ is an $F_2$-free graph  on $n\ge 7$ vertices, then
\[  \lambda (G) \le \lambda (K_{ \lfloor \frac{n}{2} \rfloor, \lceil \frac{n}{2} \rceil}^+). \]
Moreover, the equality holds if and only if $G=K_{ \lfloor \frac{n}{2} \rfloor, \lceil \frac{n}{2} \rceil}^+$.
\end{theorem}

Theorem \ref{thm-n-F2} implies Theorem \ref{EFGG-F2}, and the condition $n\ge 7$ in Theorem \ref{thm-n-F2} is tight;  
see \cite{2022LLP} for the details.   
Furthermore, one may ask how many copies of the bowtie must have in a graph with the spectral radius larger than
that of the extremal graph.  
As far as we know, 
there are very few results on spectral counting substructures, 
 including only cliques \cite{BN2007jctb} and triangles \cite{NZ2021}. 
Moreover, Ning and Zhai \cite{NZ2021}
proposed a general problem on finding 
spectral counting result for all color-critical graphs.  
In graphs with fixed order, 
the spectral supersaturation problem 
has {\bf not} been considered for substructures  
that are {\bf not} color-critical. 
Therefore, it is interesting to obtain a first result of this type.   

 \begin{figure}[htbp]
\centering
\includegraphics[scale=0.8]{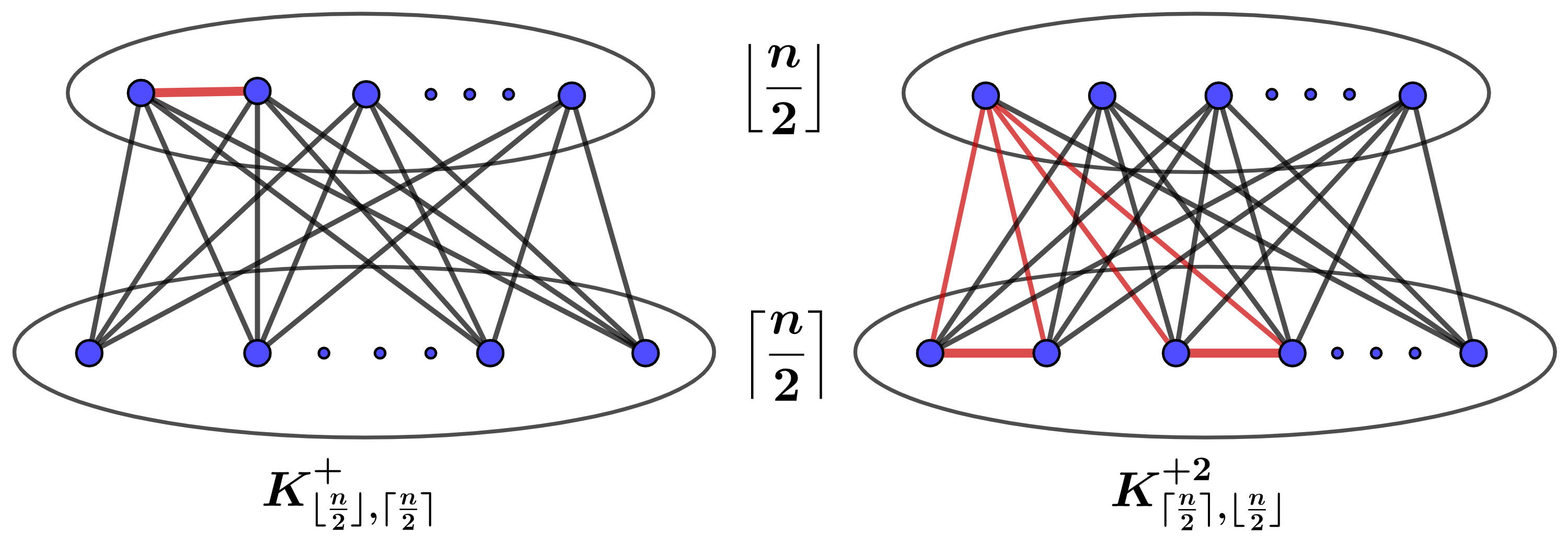} 
\caption{The extremal graphs in Theorems \ref{thm-n-F2} and  \ref{thm-bowtie}.}
\label{fig-thm6-7}
\end{figure}

\medskip 
The second main result of this paper is to 
count the number of bowties, which 
resolves an open problem proposed by Li, Lu and Peng  \cite[Problem 4.2]{2022LLP}.   
Let $K_{\lceil \frac{n}{2} \rceil, \lfloor \frac{n}{2} \rfloor}^{+2}$ be 
 the graph obtained from $K_{\lceil \frac{n}{2} \rceil, \lfloor \frac{n}{2} \rfloor}$ by adding two disjoint edges  
to the vertex part of size $\lceil \frac{n}{2} \rceil$.

\begin{theorem} \label{thm-bowtie}
If $n\ge 8.8\times 10^{6}$ and 
$G$ is an $n$-vertex graph with 
\[ \lambda (G) \ge \lambda (K_{\lceil \frac{n}{2} \rceil, \lfloor \frac{n}{2} \rfloor}^{+2}) , \]  
then $G$ has at least $\lfloor \frac{n}{2} \rfloor$  bowties, 
and $K_{\lceil \frac{n}{2} \rceil, \lfloor \frac{n}{2} \rfloor}^{+2}$ is the unique spectral extremal graph. 
\end{theorem}

The extremal graph results that count the non-color-critical substructures seem much more difficult to obtain.  
At the first glance, 
one may consider the $4$-cycle as a starting point. 
It was shown in \cite{Niki2007laa2,ZW12} 
that the spectral extremal graph for the $4$-cycle is the friendship graph 
(or possibly hanging an edge). 
Unfortunately, it is trivial to count $4$-cycles 
in an $n$-vertex graph in terms of the spectral radius,  
since adding one more edge to the extremal graph leads to at most two 
copies of the $4$-cycle.  
Therefore, the bowtie could be a natural candidate as a 
simplest non-color-critical substructure. 

\medskip 

We would like to emphasize that 
the spectral conditions in our results are {\bf weaker} than 
the conventional conditions on the size of a graph, since  
any graph $G$ with $e(G)>  e(T_{n,2})$ 
must satisfy $\lambda (G)> \lambda (T_{n,2})$.  
Otherwise, assuming that $\lambda (G) \le \lambda (T_{n,2})$, 
it follows that $e(G)\le \lfloor \frac{n}{2} \lambda (G)\rfloor \le 
\lfloor \frac{n}{2} \lambda (T_{n,2})\rfloor = e(T_{n,2})$.   
Moreover, there are many graphs with $\lambda (G)> \lambda (T_{n,2})$ 
but $e(G)< e(T_{n,2})$. 
So, the spectral condition in Theorem \ref{thm-n-2} is weaker than 
the size condition in Theorem \ref{thm-XK}. 
Observe that all extremal graphs in Figure \ref{fig-six} have exactly $\lfloor n^2/4\rfloor$ edges. Thus, Theorem \ref{thm-n-2} 
can imply Theorem \ref{thm-XK}. 
Similarly, the same implication holds for Theorem \ref{thm-bowtie},   
since every graph with $e(G)> e(T_{n,2}) +2$ satisfies 
$ \lambda (G) > \lambda (K_{\lceil \frac{n}{2} \rceil, \lfloor \frac{n}{2} \rfloor}^{+2})$. 
Incidentally, there exist many graphs $G$ that satisfy $e(G)\ge e(T_{n,2})$, 
but it is not necessary that $\lambda (G)\ge \lambda (T_{n,2})$. 
For example, 
taking an odd $n\ge 5$,  we define  
$G=K_{\frac{n-1}{2}, \frac{n+1}{2}}^{+\,\,|}$ as the graph obtained  from 
$K_{ \frac{n-1}{2},  \frac{n+1}{2}}$ by adding an edge $e_1$ to the part of size $ \frac{n-1}{2} $, and then deleting an edge $e_2$ between two parts so that $e_2$ is incident to $e_1$. 
It is not difficult to verify that $e(K_{\frac{n-1}{2}, \frac{n+1}{2}}^{+\,\,|}) = 
e(T_{n,2})$, while 
$\lambda (K_{\frac{n-1}{2}, \frac{n+1}{2}}^{+\,\,|}) < \lambda (T_{n,2})$.

\subsection{Our approach, organization and nonation}

{\bf Our approach.} 
We point out that 
the original motivation is to establish Theorems \ref{thm-n-2} 
and \ref{thm-bowtie}, since the extremal graphs in these problems 
are constructed from a bipartite Tur\'{a}n graph $T_{n,2}$ 
by embedding two disjoint edges.  
Our proofs of Theorems \ref{thm-n-2} 
and \ref{thm-bowtie} are quite different from those of Theorems \ref{thmNZ2021} and \ref{thm-n-F2}, 
the latter of which employ the Perron eigenvector 
in conjunction with the $2$-length walks starting from the vertex corresponding to the maximum coordinate of the eigenvector.   
The main idea in the proofs of our paper is based on the use of  
the supersaturation-stability method \cite{LFP2024-triangular}. 
Roughly speaking, every graph with few triangles can be made bipartite 
by removing a small number of  edges; 
see, e.g., Theorem \ref{thm-far-bipartite}.  
This seems to be of independent interest. 
To prove Theorem \ref{thm-bowtie}, 
 the spectral techniques due to 
Cioab\u{a}, Feng, Tait and Zhang \cite{CFTZ20} 
will also be applied for our purpose. 
As was the case in above-mentioned works \cite{Mub2010, PY2017, LM2022-Erd-Rad, KMP2020}, 
it is a traditional component to use the celebrated graph removal lemma 
to deal with the extremal graph problems on counting substructures. 
It is an advantage that all the proofs of our results make no use of the graph removal lemma,  
so that we do not require the order $n$ to be sufficiently large. 
To avoid unnecessary computations, 
we did not make much effort to optimize the bound on the order of the spectral extremal graph in our arguments. 
Moreover, we are optimistic that the supersaturation-stability method used in our paper 
may be useful in obtaining the spectral results for counting other graphs (even for non-color-critical graphs such as the wheel graph $W_5$ and the Petersen graph among others).

\medskip 
\noindent 
{\bf Organization.}   
 In Section \ref{sec3}, 
we shall introduce some preliminaries, 
including a spectral triangle counting lemma, 
the Moon--Moser inequality and the supersaturation-stability. Moreover, some 
 computations on the spectral radii of extremal graphs will be provided 
 in this section.  
 In Section \ref{sec4}, 
 we present the proof of Theorem  \ref{thm-n-2}. 
It is quite complicated to establish the supersaturation 
 for bowties just as presented in \cite{KMP2020}.  
 Our proof of Theorem \ref{thm-bowtie} is complicated as well.  
So the main ideas and the detailed proof will be provided separately in Section \ref{sec6}. 
In addition, we shall conclude some 
spectral extremal problems for interested readers involving 
 counting triangles and bowties in Section \ref{sec7}. 

\medskip

\noindent 
{\bf Notation.} 
We usually write $G=(V,E)$ for a simple undirected 
graph with vertex set 
$V=\{v_1,\ldots ,v_n\}$ and edge set $E=\{e_1,\ldots ,e_m\}$, where $n=|V|$ and $m=|E|$.   
If $S\subseteq V$ is a subset of the vertex set, then 
$G[S]$ denotes the subgraph of $G$ induced by $S$, 
i.e., the graph on $S$ whose edges are those edges of $G$ 
with both endpoints in $S$. 
For simplicity, we denote $e(S)=e(G[S])$ 
and we use $e(G)$ to denote $e(G[V])$.     
Let $G \setminus S$ denote the subgraph induced by $V\setminus S$.  
In addition, we write $G[S,T]$ for the bipartite subgraph of $G$ whose edges have one endpoint in $S$ and the other in $T$,   and similarly, we write $e(S,T)$ for the number of edges of $G[S,T]$.  
Let $N(v)$ be the set of vertices adjacent to a vertex $v$,  and let $d(v)=|N(v)|$ be the degree of $v$. 
Moreover, we write $N_S(v)= N(v) \cap S$ 
for the set of neighbors of $v$ in $S$, 
and we denote $d_S(v)=|N_S(v)|$ for simplicity. 
Let $t(G)$ be the number of triangles in $G$.

\section{Preliminaries}
\label{sec3}

\subsection{A spectral triangle counting lemma}

To begin with, we introduce the following spectral result for counting triangles.  

\begin{lemma}[See \cite{BN2007jctb,CFTZ20,NZ2021}] 
\label{thm-BN-CFTZ-NZ}
Let $G$ be a graph with $m$ edges. Then 
\begin{equation*} \label{eq-spectral-super-triangle}
t(G) \ge \frac{\lambda \bigl(\lambda^2 - m\bigr)}{3}. 
  \end{equation*}
  The equality holds if and only if $G$ is a complete bipartite graph. 
\end{lemma}

This result goes back to an old theorem of Nosal \cite{Nosal1970}, which asserts that every graph $G$ with $m$ edges and $\lambda (G)> \sqrt{m}$ contains a triangle; see \cite{Niki2002cpc,Ning2017-ars} for related extensions.  
The inequality in Lemma \ref{thm-BN-CFTZ-NZ} can be recast as three different versions: 
\begin{equation*} 
   \lambda^3 \le 3t + m \lambda 
\quad \Leftrightarrow \quad {t\ge \frac{1}{3}\lambda (\lambda^2 - m)  } 
\quad \Leftrightarrow \quad m \ge \lambda^2- \frac{3t}{\lambda}. 
\end{equation*}
Lemma \ref{thm-BN-CFTZ-NZ} was initially studied by 
Bollob\'{a}s and Nikiforov  \cite{BN2007jctb}, 
and it was recently revisited by Cioab\u{a},  Feng, 
Tait and Zhang \cite{CFTZ20} as well as Ning and Zhai \cite{NZ2021}. 
Crucially, Lemma \ref{thm-BN-CFTZ-NZ}  tells us that if 
 $\lambda (G)\ge {n}/{2}$ and $t(G)\le kn$, then $G$ has at least $n^2/4 - 6k$ edges. 
This provides the expected number of edges  
 for forcing a graph to be close to bipartite.

\subsection{The Moon--Moser inequality}

Next, we present a result of Moon and Moser \cite[p. 285]{MM1962}.  

\begin{theorem}[Moon--Moser] 
\label{thm-MM-k3}
Let $G$ be a graph on $n$ vertices with $m$ edges. 
Then  
\[ t(G) \ge \frac{4m}{3n}\left(m- \frac{n^2}{4} \right). \]
\end{theorem}

Two proofs of Theorem \ref{thm-MM-k3} 
can be found in \cite[p. 297]{Bollobas78} and \cite[p. 443]{Lov1979}. 
For convenience of the readers, we provide a different proof as follows. 
 
\begin{proof}  
For each $v\in V(G)$, we denote $N_v= N(v)$ and $N_v^c = V(G) \setminus N_v$. Then 
\[ \sum_{w\in N_v} d(w) = 2e(N_v) + e(N_v, N_v^c) = 
e(N_v) + m - e(N_v^c). \]    
Note that  $3t(G) =\sum_{v\in V}e(N_v)$.  
Summing the  above equalities, we get 
\[    \sum_{v\in V} \sum_{w\in N_v} d(w)  =
 3t(G) + mn  - 
\sum_{v\in V} e(N_v^c). \]
On the other hand, 
we obtain from Cauchy--Schwarz's inequality that  
\[   \sum_{v\in V} \sum_{w\in N_v} d(w)  
= \sum_{w\in V} d(w)^2 \ge \frac{4m^2}{n}.  \] 
Combining the two inequalities, 
we have 
$ 3t(G) \ge \frac{4m^2}{n} - mn $, as needed. 
\end{proof}

Theorem \ref{thm-MM-k3} 
gives a lower bound on the number of triangles in a graph 
with size greater than $n^2/4$.
As a byproduct of this result, 
we see that an $n$-vertex graph with $n^2/4 +1$ edges 
contains at least $n/3$ triangles. Moreover, 
we can obtain the following supersaturation for triangles, 
which will be used in our proof of Theorem \ref{thm-bowtie}. 

\begin{corollary} \label{cor-MM}
If $\varepsilon >0$ and $G$ is a graph 
on $n$ vertices with 
\[  e(G) \ge \frac{n^2}{4} + \varepsilon n^2,  \]
then $G$ contains more than $\frac{\varepsilon}{3}n^3$ triangles. 
\end{corollary}

\subsection{The supersaturation-stability}

For a real number $\varepsilon >0$, 
we say that a graph $G$ is {\it $\varepsilon$-far from being bipartite}  
if  $G'$ is not bipartite 
for every subgraph $G'$ of $G$ with 
$e(G')>  e(G) - \varepsilon$. 
That is to say, if $G$ is $\varepsilon$-far from being bipartite, 
then no matter how we delete less than $\varepsilon$ edges from $G$, 
the resulting graph is always non-bipartite. 
Consequently, we must remove at least $\varepsilon$ edges from 
$G$ to make it bipartite.  
The following theorem was proved by Balogh et al. \cite{BBCLMS2017}.

\begin{theorem}[See \cite{BBCLMS2017}] \label{thm-far-bipartite} 
If $G$ is $\varepsilon$-far from being bipartite, then 
\[ t(G) \ge \frac{n}{6}\left( m + \varepsilon -\frac{n^2}{4}\right).  \] 
\end{theorem}

Theorem \ref{thm-far-bipartite} can imply the supersaturation on triangles 
for graphs with less than $n^2/4$ edges. 
Namely, if $G$ is an $n$-vertex graph with 
$n^2/4 - q$ edges, and $G$ has less than $kn$ triangles, then 
$G$ is not $(6k +q)$-far from being bipartite. 
Consequently, we can remove less than 
$6k+q$ edges from $G$ to make it bipartite, i.e., $G$ has a large bipartite subgraph with more than $n^2/4 - 2q-6k$ edges. 
This result enables us to avoid the use of the triangle removal lemma 
so that we can get rid of the condition where the order $n$ is sufficiently large.  
We refer to 
\cite{LFP2024-triangular} for some related applications. 
As a matter of fact, finding a large bipartite subgraph 
plays an important role in the proofs of   
Theorems \ref{thm-n-2} and \ref{thm-bowtie}.

\subsection{Computations for extremal graphs}

The extremal graphs in our problem are 
not symmetric, 
and a lot of calculations on the spectral radii of graphs are needed. 
We give the following basic lemmas without proofs. 

\begin{lemma} \label{lem-root}
Suppose that $f(x)$ is a real polynomial function  with the largest root 
$\gamma$ in 
$(\frac{n}{2}-1, \frac{n}{2}+1)$, and $f(x)$ increases monotonically in $(\frac{n}{2}-1, \frac{n}{2}+1)$. 
If $x_0\in (\frac{n}{2}-1, \frac{n}{2}+1)$ such that $f(x_0) >0$, then $\gamma < x_0$; 
if $f(x_0) <0$, then $x_0< \gamma$. 
\end{lemma}

\begin{lemma} \label{lem-roots}
Let $f(x)$ and $g(x)$ be real polynomial functions with largest roots $\beta$ and $\gamma$, respectively.  
If  $\beta, \gamma \in (\frac{n}{2}-1, \frac{n}{2} +1)$, 
 $f(x)$ and $g(x)$ are increasing monotonically,  
and $f(x)> g(x)$ for any $x\in (\frac{n}{2}-1, \frac{n}{2} +1)$, 
then  $\beta < \gamma$. 
\end{lemma}

The next lemma gives
a lower bound on the spectral radius of 
$K_{\lceil \frac{n}{2} \rceil, \lfloor \frac{n}{2} \rfloor}^{+2}$. 

\begin{lemma} \label{lem-F-G}
(a) If $n$ is even, then $\lambda (K_{\frac{n}{2},  \frac{n}{2}}^{+2})$ is the largest root of 
\[  f(x)= x^3 - x^2 - (n^2 x)/4 + n^2/4 - 2 n. \]
(b) If $n$ is odd, then 
$\lambda (K_{ \frac{n+1}{2}, \frac{n-1}{2}}^{+2})$ is the largest root of 
\[   g(x) = x^3 - x^2 + x/4 - (n^2 x)/4 + n^2/4 - 2 n + 7/4. \]
Furthermore, we get 
\[  \lambda^2 (K_{\lceil \frac{n}{2} \rceil, \lfloor \frac{n}{2} \rfloor}^{+2}) > \lfloor n^2/4 \rfloor  +4. \]
\end{lemma}

\begin{proof}
(a) Let $\mathbf{x}=(x_1,x_2,\ldots ,x_n)^{\top}$ be the Perron eigenvector corresponding to $\lambda(K_{\frac{n}{2}, \frac{n}{2}}^{+2})$.
We partition the vertex set of $ K_{\frac{n}{2}, \frac{n}{2}}^{+2}$ as $\Pi$:
\[   V(K_{\frac{n}{2}, \frac{n}{2}}^{+2} )=
X_1 \cup X_2 \cup Y, \]
 where $X_1=\{u_1,u_2,u_3,u_4\}$ forms a matching of size two, 
$X_1\cup X_2$ and $Y$ are partite sets of $K_{\frac{n}{2}, \frac{n}{2}}$.
Compared with the neighborhoods, we can see that $x_{u_1}= \cdots = x_{u_4}$,
all the coordinates of the vector $\mathbf{x}$
corresponding to the vertices of $X_2$  are equal
(the coordinates of the vertices of $Y$ are equal).
Without loss of generality, we may assume that
$x_{u_1}=\cdots =x_{u_4}=a$, $x_u=b$ for each $u\in X_2$, and $x_v=c$
for each $v\in Y$. Then
\[  \begin{cases}
\lambda a= a + \frac{n}{2} c, \\
\lambda b = \frac{n}{2}c, \\
\lambda c = 4a + (\frac{n}{2}-4)b.
\end{cases} \]
Thus, $\lambda(K_{\frac{n}{2}, \frac{n}{2}}^{+2})$ is the largest eigenvalue of
\[  B_{\Pi} =  \begin{bmatrix}
1 & 0 & \frac{n}{2} \\
0 & 0 & \frac{n}{2} \\
4 & \frac{n}{2}-4 & 0
\end{bmatrix}.  \]
By calculation, we know that $\lambda(K_{\frac{n}{2}, \frac{n}{2}}^{+2})$ is the largest  root of
\[ f(x)= \det (xI_3 - B_{\Pi}) 
=x^3 - x^2 - (n^2 x)/4 + n^2/4 - 2 n.  \]
We mention that the partition $\Pi$ (into three parts) 
is called an {\it equitable partition}, and $B_{\Pi}$ is called the {\it quotient matrix} of $\Pi$. 
For a connected graph with an equitable partition $\Pi$, the spectral radius of the graph equals the spectral radius of the quotient matrix $B_{\Pi}$.

(b) For odd $n$, 
we partition the vertex set of $K_{\frac{n+1}{2}, \frac{n-1}{2}}^{+2}$
as $\Pi'$:
\[  V(K_{\frac{n+1}{2}, \frac{n-1}{2}}^{+2}) = X_1\cup X_2\cup Y, \]
where $X_1=\{u_1,u_2,u_3,u_4\}$ spans 
two disjoint edges, $X_1\cup X_2$ and $Y$
are partite sets of $K_{\frac{n+1}{2}, \frac{n-1}{2}}$
satisfying $|X_1| + |X_2|= \frac{n+1}{2}$
and $|Y|=\frac{n-1}{2}$. Then
a similar argument yields that
$\lambda (K_{\frac{n+1}{2}, \frac{n-1}{2}}^{+2})$ is the largest eigenvalue of
\[  B_{\Pi'}= \begin{bmatrix}
1 & 0 & \frac{n-1}{2} \\
0 & 0 & \frac{n-1}{2} \\
4 & \frac{n+1}{2}-4 & 0
\end{bmatrix}.  \]
Thus, $\lambda (K_{\frac{n+1}{2}, \frac{n-1}{2}}^{+2})$
is the largest root of
\[  g(x)= \det (xI_3- B_{\Pi'})
=x^3 - x^2 + x/4 - (n^2 x)/4 + n^2/4 - 2 n + 7/4.  \] 
Moreover, it is easy to check that 
\[ f(\sqrt{n^2/4 +4}) = 2 (-2 - n + \sqrt{16 + n^2}) <0  \]
and 
\[ g(\sqrt{(n^2-1)/4 +4}) 
=2 (-1 - n + \sqrt{15 + n^2}) <0. \]
By Lemma \ref{lem-root}, we get 
$ \sqrt{\lfloor n^2/4 \rfloor  +4} < 
\lambda (K_{\lceil \frac{n}{2} \rceil, \lfloor \frac{n}{2} \rfloor}^{+2}) $. 
This completes the proof.
\end{proof}

The following two lemmas will be used in the proof of Theorem \ref{thm-n-2}. 

\begin{lemma} \label{lem-plus2}
Suppose that $n\ge 20$ and 
$G$ is a subgraph of $K_{s,t}^{+2}$, where 
$\lfloor {n}/{2} \rfloor -2 \le s 
\le \lceil {n}/{2} \rceil +2$ and $t=n-s$. 
If $G$ has at most $n-4$ triangles, then 
$\lambda (G)< \lambda (T_{n,2})$. 
\end{lemma}

\begin{proof}
We prove the case $n$ even here, and 
 the case of $n$ odd can be proved 
 in a similar manner.   
If $s= \frac{n}{2} -2$, 
then $G$ is obtained from $K_{\frac{n}{2}-2, \frac{n}{2} +2}^{+2}$ 
by deleting at least $8$ cross-edges to destroy at least $8$ triangles.  Actually, one can compute that deleting any cross-edge yields a graph $L_1$ whose spectral radius is less than $\frac{n}{2}$; see Figure \ref{fig-L1-L6}.  
By the equitable partition analogous to the proof of  Lemma \ref{lem-F-G}, we know that $\lambda (L_1)$ is the largest root of 
\begin{align*}
 \ell_1(x) &:=17 x + 4 n x - (n^2 x)/2 - 9 x^2 - n x^2 + (n^2 x^2)/2 - 10 x^3 \\ 
& \quad -  2 n x^3 + (n^2 x^3)/4 + 4 x^4 - (n^2 x^4)/4 - x^5 + x^6. 
\end{align*}
Note that $\ell_1({n}/{2})=1/64 (544 n - 16 n^2 - 112 n^3 + 8 n^4) >0$. So we get $\lambda (L_1) < {n}/{2}=\lambda (T_{n,2})$. 

 \begin{figure}[H]
\centering
\includegraphics[scale=0.8]{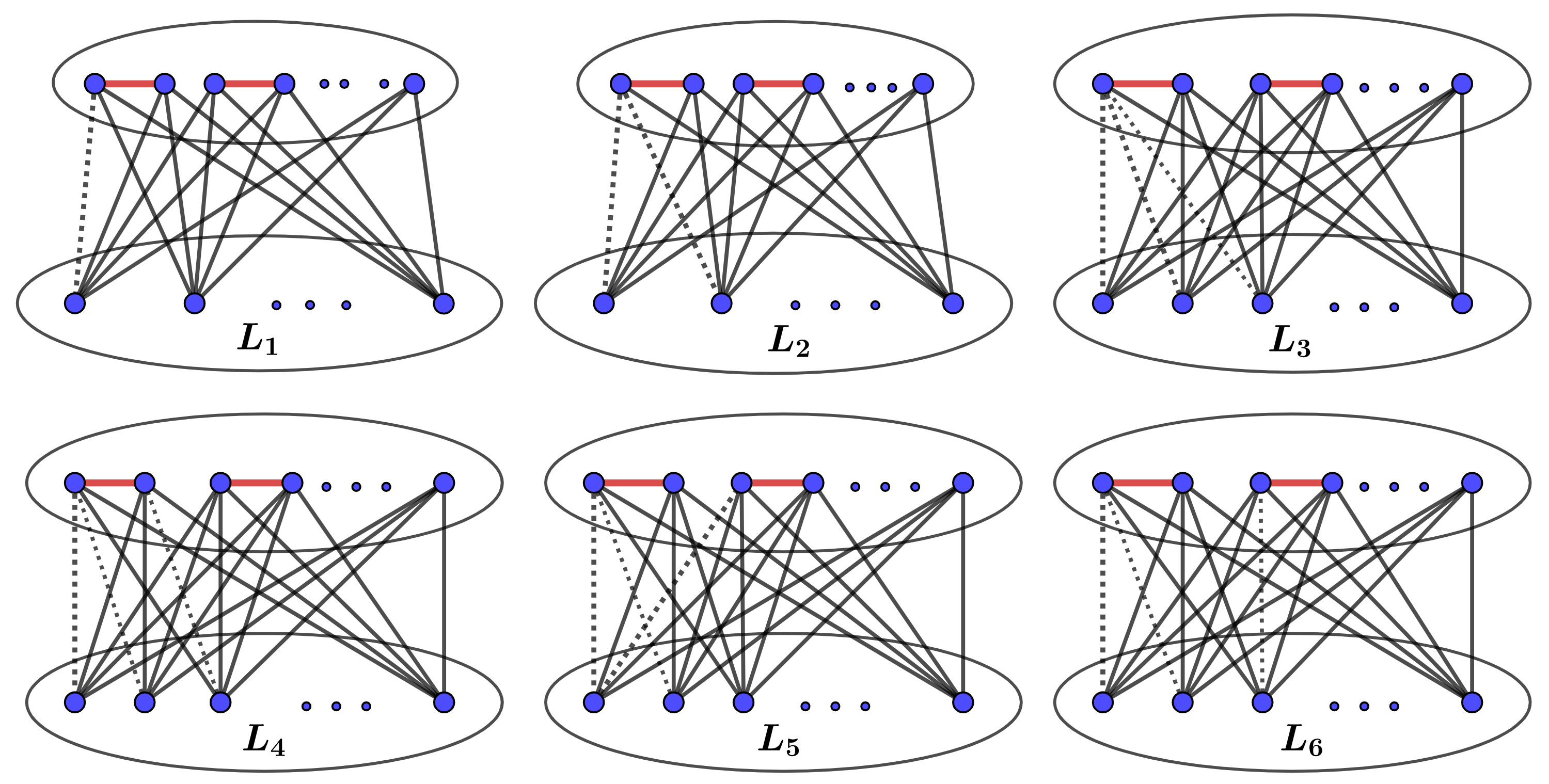} 
\caption{The graphs $L_1,\ldots ,L_6$.}
\label{fig-L1-L6}
\end{figure}

If $s=\frac{n}{2} -1$, then $G$ is obtained from 
$K_{\frac{n}{2}-1,\frac{n}{2}+1}^{+2}$ by deleting 
at least $6$ cross-edges in order to destroy at least $6$ triangles. 
Observe that there are at least two missing edges incident to 
the same vertex of a class-edge in the part $S$. 
Therefore, $G$ is a subgraph of $L_2$ in Figure \ref{fig-L1-L6}. 
 By calculation, we get that $\lambda (L_2)$ 
 is the largest root of 
\begin{align*}
  \ell_2(x)&:= 
  -3 x + 7 n x - (3 n^2 x)/4 - x^2 - 3 n x^2 + (3 n^2 x^2)/4 - 2 x^3 \\ 
  & \quad -  2 n x^3 + (n^2 x^3)/4 + 2 x^4 - (n^2 x^4)/4 - x^5 + x^6 .
 \end{align*} 
 Then  $\ell_2(n/2) = 
 1/64 (-96 n + 208 n^2 - 88 n^3 + 4 n^4) >0$ for $n\ge 20$, which yields 
 $\lambda (L_2) < n/2$. 
 
 If $s=\frac{n}{2}$, then $G$ is constructed from 
 $K_{\frac{n}{2},\frac{n}{2}}^{+2}$ by deleting at least $4$ cross-edges to remove at least $4$ triangles. 
 In fact, as long as we delete $3$ cross-edges (to remove $3$ triangles), we can show that 
 the resulting graph has spectral radius less than $\frac{n}{2}$; 
 see the graphs $L_3,L_4,\ldots,L_8$. 
Upon computation, we obtain that $\lambda (L_3)$ is the largest root of 
 \begin{align*}
  \ell_3(x)&:= 
  -27 x + 11 n x - n^2 x + 3 x^2 - 6 n x^2 + n^2 x^2 + 4 x^3 \\ 
  & \quad -  2 n x^3 + (n^2 x^3)/4 + 2 x^4 - (n^2 x^4)/4 - x^5 + x^6. 
   \end{align*}  
 It gives $\ell_3(n/2) = 1/64 (-864 n + 400 n^2 - 96 n^3 + 8 n^4) >0$. Thus $\lambda (L_3) < n/2$. 
 Similarly, we can compute that 
 $\lambda (L_4)$ is the largest root of 
  \begin{align*}
   \ell_4(x) &:= 24 - 7 n + n^2/2 - 28 x + 7 n x - (n^2 x)/2 - 5 x^2 + 7 n x^2 - 
 n^2 x^2 - 3 x^3 \\ 
 & \quad - 4 n x^3 + n^2 x^3 + 4 x^4 - 2 n x^4 + (
 n^2 x^4)/4 + 2 x^5 - (n^2 x^5)/4 - x^6 + x^7. 
    \end{align*}   
 Since $\ell_4(n/2)= 
 1/128 (3072 - 2688 n + 352 n^2 + 144 n^3 - 64 n^4 + 8 n^5) >0$, 
 we obtain $\lambda (L_4) < n/2$. 
 For the graph $L_5$, we compute that 
 $\lambda (L_5)$ is the largest root of 
 \begin{align*} 
  \ell_5(x)&:=  -15 x^2 + 10 n x^2 - (5 n^2 x^2)/4 - 10 x^3 + 3 n x^3 + 6 x^4  \\ 
  &\quad - 7 n x^4 + (5 n^2 x^4)/4 + 6 x^5 - 2 n x^5 + x^6 - (n^2 x^6)/4 + x^8. 
  \end{align*}
 It is easy to check that 
 $\ell_5(n/2) = 1/256 (-960 n^2 + 320 n^3 + 112 n^4 - 64 n^5 + 8 n^6) >0$. Then we get $\lambda (L_5)< n/2$. 
 For the graph $L_6$, we can obtain that $\lambda (L_6)$ is the largest root of 
  \begin{align*} 
   \ell_6(x) &:=  -10 x^2 + 7 n x^2 - (3 n^2 x^2)/4 - 11 x^3 + 4 n x^3 + 2 x^4 \\ 
   & \quad  - 5 n x^4 + n^2 x^4 + 5 x^5 - 2 n x^5 + x^6 - (n^2 x^6)/4 + x^8. 
     \end{align*} 
    Since $\ell_6(n/2) = 1/256 (-640 n^2 + 96 n^3 + 112 n^4 - 40 n^5 + 4 n^6)>0$, we have $\lambda (L_6)< n/2$. 
   
     \begin{figure}[H]
\centering
\includegraphics[scale=0.8]{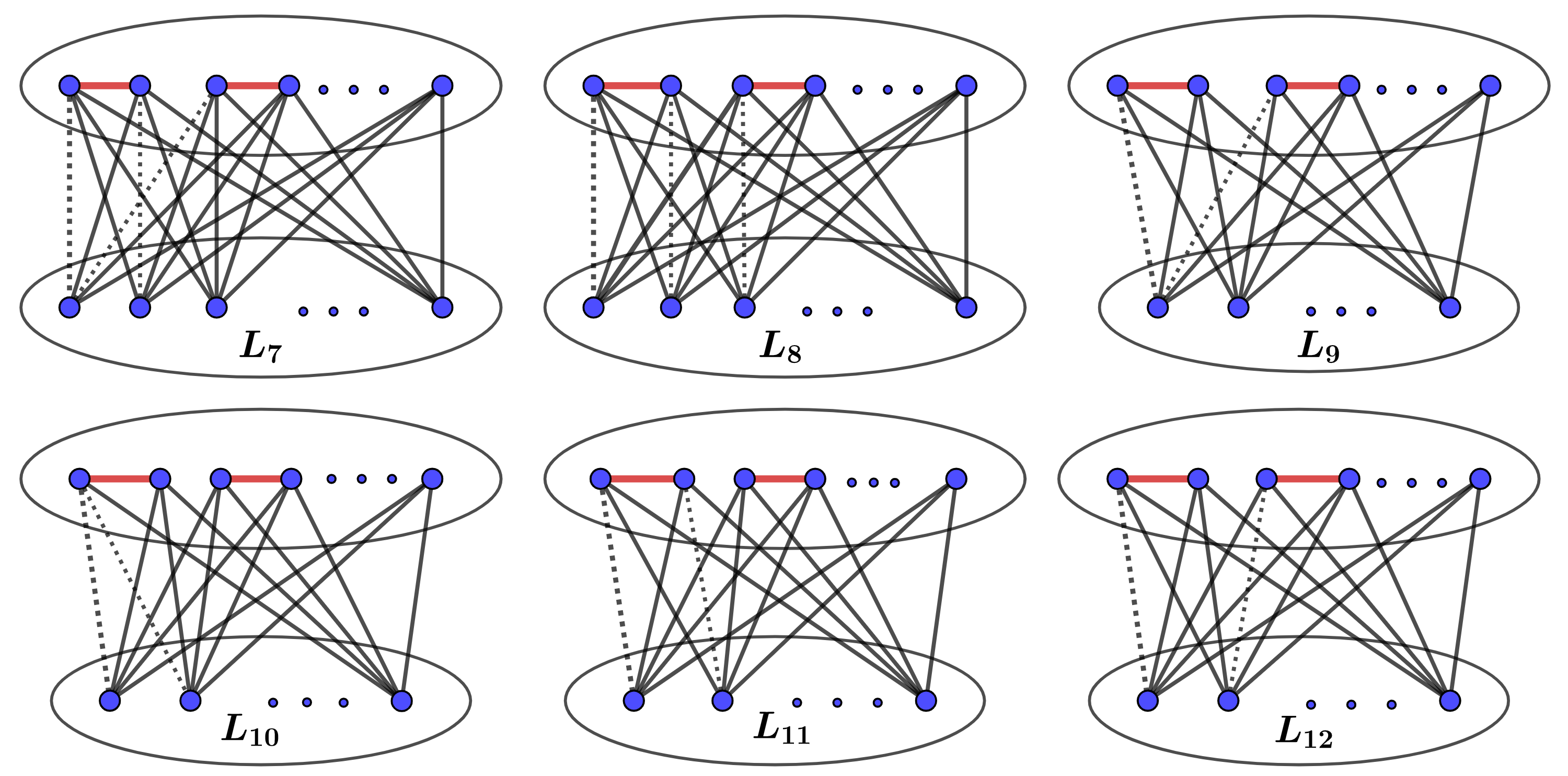} 
\caption{The graphs $L_7,\ldots ,L_{12}$.}
\label{fig-L7-L12}
\end{figure} 

\noindent 
    Similarly, for the graph $L_7$, 
    it is easy to verify that $\lambda (L_7)$ is the largest root of 
      \begin{align*} 
   \ell_7(x) &:= 
   8 - 3 n + n^2/4 - 8 x + n x - 18 x^2 + 12 n x^2 - (3 n^2 x^2)/2 - 
 2 x^3 + 2 n x^3 \\ 
 & \quad  + x^4  - 6 n x^4 + (5 n^2 x^4)/4 + 6 x^5 - 
 2 n x^5 + x^6 - (n^2 x^6)/4 + x^8, 
\end{align*} 
and $\ell_7(n/2) = 1/256 (2048 - 1792 n - 960 n^2 + 704 n^3 - 16 n^4 - 48 n^5 + 8 n^6)>0$. Then we obtain $\lambda (L_7)< n/2$. 
For the graph $L_8$, we get that $\lambda (L_8)$ 
is the largest root of 
\begin{align*} 
   \ell_8(x) &:= -17 x + 6 n x - (n^2 x)/2 - 5 x^2 + 5 n x^2 - (n^2 x^2)/2 - x^3 - 
 3 n x^3 \\
 &\quad + (3 n^2 x^3)/4  + 2 x^4 - 2 n x^4 + (n^2 x^4)/4 + 3 x^5 - (
 n^2 x^5)/4 - x^6 + x^7. 
 \end{align*} 
Note that $\ell_8(n/2) = 1/128 (-1088 n + 224 n^2 + 112 n^3 - 48 n^4 + 8 n^5)>0$. 
It follows that $\lambda (L_8) < n/2$. 
Therefore, we checked all possible graphs in the case $s=\frac{n}{2}$. 

If $s=\frac{n}{2} +1$, then $G$ is obtained from $K_{\frac{n}{2} +1, \frac{n}{2} -1}^{+2}$ by deleting at least $2$ edges to remove 
at least $2$ triangles. 
In this case, it is easy to obtain the following graphs in Figure \ref{fig-L7-L12}. Firstly, we consider the graph $L_9$ whose spectral radius is the largest root of 
\begin{align*} 
   \ell_9(x) :=7 x - 5 n x + (3 n^2 x)/4 + 8 x^2 - 2 n x^2 + 2 x^3 - (
 n^2 x^3)/4 + x^5.  
\end{align*} 
Since $\ell_9(n/2)= 1/32 (112 n - 16 n^2 + 4 n^3) >0$, 
we get $\lambda (L_9) < n/2$. Similarly, we denote 
\begin{align*} 
\ell_{10}(x) &:= -15 x + 7 n x - (3 n^2 x)/4 - 5 x^2 - 3 n x^2 + (3 n^2 x^2)/4  \\ 
& \quad + 6 x^3 - 2 n x^3 + (n^2 x^3)/4 + 2 x^4 - (n^2 x^4)/4 - x^5 + x^6, 
 \end{align*} 
and $\lambda (L_{10})$ is the largest root of $\ell_{10}(x)$. 
Note that $\ell_{10}(n/2) = 1/64 (-480 n + 144 n^2 - 24 n^3 + 4 n^4) >0$. Then $\lambda (L_{10}) < n/2$. 
For the graph $L_{11}$, we get that $\lambda (L_{11})$ 
is the largest root of 
\begin{align*} 
\ell_{11}(x) &:= -9 + 3 n - n^2/4 - x^2 - 2 n x^2 + (n^2 x^2)/2  + 5 x^3 - (n^2 x^3)/4 - 2 x^4 + x^5. 
 \end{align*} 
 Since $\ell_{11}(n/2)= 1/32 (-288 + 96 n - 16 n^2 + 4 n^3)>0$, 
 we have $\lambda (L_{11}) < n/2$. 
By calculation, we know that $\lambda (L_{12})$ is the largest 
root of 
\begin{align*} 
\ell_{12}(x) &:= 6 x - 4 n x + (n^2 x)/2 + 8 x^2 - 2 n x^2 + 3 x^3 - (n^2 x^3)/4 + x^5,  
 \end{align*}  
 and $\ell_{12}(n/2) = 1/32 (96 n + 4 n^3) >0$. 
Thus, it follows that $\lambda (L_{12}) < n/2$. 
 
 If $s=\frac{n}{2} +2$, then $G$ is a subgraph of 
 $K_{\frac{n}{2} +2, \frac{n}{2} -2}^{+2}$. 
 By direct computation, we obtain that 
 $\lambda (K_{\frac{n}{2} +2, \frac{n}{2} -2}^{+2})$ 
 is the largest root of $\ell_{13}(x):=4 - 2 n + n^2/4 + 4 x - (n^2 x)/4 - x^2 + x^3$ and $\ell_{13}(n/2) = 4>0$. 
It follows that $\lambda (G)\le \lambda (K_{\frac{n}{2} +2, \frac{n}{2} -2}^{+2}) < n/2$, as needed. 
\end{proof}

From the above computations, we can observe an interesting phenomenon.  
All graphs in Figures \ref{fig-L1-L6} and \ref{fig-L7-L12} 
have at most $e(T_{n,2}) -1$ edges and the spectral radius less than $\lambda (T_{n,2})$. 
Recall that the graphs in Figure \ref{fig-six} have exactly $e(T_{n,2})$ edges, while the spectral radius is greater than $\lambda (T_{n,2})$. 
Roughly speaking, for such nearly regular graphs, 
the spectral radius is very close to the average degree.

Let $K_{s,t}^{++}$ denote the graph formed by adding exactly one additional edge to each partite set of the complete bipartite graph $K_{s,t}$.   
The following lemma is similar to Lemma \ref{lem-plus2}. 
The proof proceeds by direct computations, so we omit the details. 

\begin{lemma} \label{lem-plus1}
Suppose that $n\ge 113$ and $G$ is a subgraph of $K_{s,t}^{++}$, where 
$\lfloor {n}/{2} \rfloor -2 \le s 
\le \lceil {n}/{2} \rceil +2$ and $t=n-s$. 
If $G$ has at most $n-4$ triangles, then 
$\lambda (G)< \lambda (T_{n,2})$. 
\end{lemma}

\section{Proof of Theorem \ref{thm-n-2}}
\label{sec4}

We now provide the proof of Theorem \ref{thm-n-2}.  
Assume that $G$ is an $n$-vertex graph with $\lambda (G) \ge \lambda (T_{n,2})$ and $\tau_3(G) \ge 2$. 
Suppose on the contrary  that 
$G$  has at most $n-4$ triangles. 
In the sequel, we will deduce a contradiction.  
Since $\lambda (T_{n,2})= \sqrt{\lfloor n^2/4\rfloor} $, we get 
$\lambda (G) > \frac{n-1}{2}$. 
It follows  from Lemma \ref{thm-BN-CFTZ-NZ} that 
\[   e(G)\ge \lambda^2- \frac{3t}{\lambda} 
> \lambda^2 - \frac{6t}{n-1} \ge   
\left\lfloor \frac{n^2}{4} \right\rfloor - \frac{6(n-4)}{n-1} . \]  
Note that $e(G)$ must be a positive integer. Then 
\[  e(G)\ge \left \lfloor \frac{n^2}{4} \right\rfloor  -5. \]

\begin{lemma} \label{rough-part}
There exists a partition $V(G)=S\cup T$ 
such that $e(S) + e(T)<12$. 
Furthermore, we have $e(S,T)\ge \lfloor n^2/4\rfloor -16$ 
and $n/2 -4 \le |S|, |T|\le n/2 +4$. 
\end{lemma}

\begin{proof}
 We claim that $G$ is not $12$-far from being bipartite. 
 Otherwise,  Theorem \ref{thm-far-bipartite} implies 
$ t(G)\ge \frac{n}{6} (e(G) +12 - \frac{n^2}{4} ) > n$, 
which is a contradiction with our assumption that 
$t(G)\le n-4$. Therefore, the graph $G$ is not 
$12$-far from being bipartite. In other words, there is a partition of
the vertex set of $G$ as $V(G)=S\cup T$ such that 
$e(S)+e(T)< 12$. Then 
\[  e(S,T) = e(G)-e(S) - e(T) \ge e(G) -11\ge \left\lfloor \frac{n^2}{4} \right\rfloor -16. \]    
By the AM-GM inequality, we get 
${n}/{2}-4 \le |S|,|T| \le {n}/{2} +4$. 
\end{proof}

\begin{lemma} \label{lem-S+T=2}
We have that $e(S) +e(T) =2$ and the two edges of $G[S]\cup G[T]$ are disjoint. 
\end{lemma} 

\begin{proof} 
The condition $\tau_3(G)\ge 2$ means that $G$ contains at least one triangle. 
Then we have $e(S) +e(T)\ge 1$. 
Suppose on the contrary that $e(S) + e(T)\neq 2$. 
Firstly, if $e(S)=1$ or $e(T)=1$, then 
all triangles in $G$ can be  covered by one vertex, 
which contradicts with $\tau_3(G)\ge 2$. 
Now, suppose that $e(S)+ e(T)=s$, where $3\le s\le 11$. 
Observe that each missing cross-edges between $S$ and $T$ 
is contained in at most $s$ triangles. Then  
we have $ t(G) \ge s (\frac{n}{2} -4) - 16s > n-3$  for $n\ge 113$, 
which is a contradiction.  Thus, we must have 
$ e(S) + e(T)=2$. 
Furthermore, the two class-edges in $G[S]\cup G[T]$ are 
vertex-disjoint in $G$. Otherwise, there is a vertex that covers all triangles of $G$, contradicting with $\tau_3(G)\ge 2$.  
\end{proof}

To save calculations, we give a refinement on Lemma \ref{rough-part}. 

\begin{lemma} \label{lem-pm2}
We have $e(S,T)\ge \lfloor n^2/4\rfloor -7$ and 
$\lfloor n/2 \rfloor - 2 \le |S|, |T|\le \lceil n/2\rceil +2$. 
\end{lemma}

\begin{proof}
In contrast to the proof of Lemma \ref{rough-part}, 
which heavily relied on Theorem \ref{thm-far-bipartite}, 
our current refinement no longer needs to invoke Theorem \ref{thm-far-bipartite}. 
By Lemma \ref{lem-S+T=2}, we get $e(S)+e(T)=2$, and then   
\[   e(S,T)= e(G) - e(S)- e(T) 
\ge \left\lfloor \frac{n^2}{4} \right\rfloor -7. \]  
Thus, it follows that 
$ \lfloor {n}/{2} \rfloor -2 \le |S|,|T| 
\le \lceil {n}/{2} \rceil +2$, as desired. 
\end{proof}

From Lemmas \ref{lem-S+T=2} and \ref{lem-pm2},  
we may assume by symmetry that either $e(S)=2$ or $e(S)=e(T)=1$.  
Then we get $\lambda (G)< \lambda (T_{n,2})$  
 by using Lemmas \ref{lem-plus2} 
and \ref{lem-plus1}. This contradicts with the previous assumption. 
So we complete the proof. \hfill $\square$

\section{Proof of Theorem \ref{thm-bowtie}}
\label{sec6}

\subsection{An overview of the proof}
In this section, we shall prove Theorem \ref{thm-bowtie}. 
Our solution applies the framework set up by the authors  \cite{LFP2024-triangular},  
and we need to put much effort into the analysis of the structure 
of expected extremal graph. 
We always assume that $G$ is a graph on $n$  vertices with 
$\lambda (G) \ge \lambda (K_{\lceil \frac{n}{2} \rceil, \lfloor \frac{n}{2} \rfloor}^{+2})$ 
and $G$ contains at most  $\lfloor n/2\rfloor $ bowties. 
Among such graphs, we choose $G$ as the graph with the maximum spectral radius. 
Our goal is to prove that $G$ 
is exactly 
the desired extremal graph $K_{\lceil \frac{n}{2} \rceil, \lfloor \frac{n}{2} \rfloor}^{+2}$. 
Before starting the proof in detail, we outline the main steps as follows. 

\begin{itemize}

\item[\ding{172}]
The first step is to obtain a bipartition of  vertices of $G$. 
Using a theorem of Alon and Shikhelman \cite{AS2016}, 
we will show that $G$ 
contains less than $10n^{2}$ triangles. 
Using Lemma \ref{thm-BN-CFTZ-NZ}, 
we can get $e(G) > {n^2}/{4} - 180n$. 
Applying the supersaturation-stability in Theorem \ref{thm-far-bipartite} , 
we obtain an approximate bipartition of $G$; see Lemma \ref{lem-partition}.

\item[\ding{173}]
The second step refines the structure of 
the desired extremal graph. 
We shall prove that $G$ has  few `bad' vertices that have large degree in their own part and have small degree in the other part.    
Moreover, we show 
 that $G$ is nearly a complete bipartite graph whose vertex parts are constant-sized around ${n}/{2}$; see Lemmas \ref{Lupper} --  \ref{STle3}.

\item[\ding{174}] 
Thirdly, we show that $G$ has at most 
$3({n}/{2} + 14\sqrt{n})$ triangles. 
This improves the previous result in the first step. 
Using Lemma \ref{thm-BN-CFTZ-NZ} again,  
we conclude that there exists a partition $V(G)=S\cup T$ such that 
$e(S,T)\ge \lfloor n^2/4\rfloor -8$ and ${n}/{2} -2 \le 
|S|, |T| \le {n}/{2} +2$; see Lemmas \ref{STlambdarefine} --  \ref{SorT=2}. 

\item[\ding{175}]
Finally, 
we shall prove either $e(S)=2$ or $e(T)=2$.  
With some calculations, we  determine that 
$G$ is the expected spectral extremal graph $K_{\lceil \frac{n}{2} \rceil, \lfloor \frac{n}{2} \rfloor}^{+2}$. 
\end{itemize}

\subsection{The approximate structure}

Recall that a set of $k$ pairwise disjoint edges is 
called {\it a matching of size $k$}.  
A famous theorem due to Erd\H{o}s and Gallai \cite{EG1959}  
states that if $n\ge \frac{5k+3}{2}$ and 
$G$ is an $n$-vertex graph with $e(G) > k(n-k) +{k \choose 2}$, 
then $G$ contains a  matching of size $k+1$. 

\medskip 
Assume that $G$ has less than  
$\lfloor {n}/{2}\rfloor $ bowties. 
Then we can give a bound on the number of triangles in $G$. 
Indeed, for each vertex $v\in V(G)$, 
the neighborhood $N(v)$ does not contain 
a matching of size $\sqrt{n} +1$.  
Setting $k=\sqrt{n}$ and applying the Erd\H{o}s--Gallai theorem, we get 
 $e(N(v)) < kn =n^{3/2}$. 
Using the double counting, it follows that 
$  3t(G) = \sum_{v\in V} e(N(v))  <  n^{5/2}$.  
Therefore, $G$ contains less than $\frac{1}{3}n^{5/2}$ triangles.  
In order to improve this bound, we introduce a 
result of Alon and Shikhelman \cite[Lemma 3.1]{AS2016}. 

\begin{lemma}[See \cite{AS2016}] \label{lem-AS}
Let $G$ be a graph on $n$ vertices.  
If $G$ is $F_k$-free, then 
\[ t(G)< (9k-15)(k+1)n. \]  
\end{lemma}

In what follows, we show that 
the extremal graph $G$ has an approximately correct structure. 
We can apply Lemma \ref{lem-AS} to 
show that $G$ has less than $10n^2$ 
triangles. 
Then Lemma \ref{thm-BN-CFTZ-NZ} implies 
that $G$ has more than $n^2/4 - 60n$ edges. 
So we can use Theorem \ref{thm-far-bipartite} and obtain an optimal partition of vertices of $G$.

\begin{lemma}\label{lem-partition}
There exists a vertex partition  $V(G)=S\cup T$ such that 
$$ e(S) + e(T) < 120n $$
and 
$$e(S, T)> \frac{n^2}{4} - 180n. $$
Furthermore, we have 
\[
\frac{n}{2} - 14\sqrt{n} < |S|, |T| < 
\frac{n}{2} + 14\sqrt{n}.
\]
\end{lemma}

\begin{proof}
Note that $\lambda (G)\ge  \lambda (K_{\lceil \frac{n}{2} \rceil, \lfloor \frac{n}{2} \rfloor}^{+2}) > \frac{n}{2}$.  
By Lemma \ref{lem-AS}, we obtain 
\[  t(G)< 10n^2. \]
Then Lemma \ref{thm-BN-CFTZ-NZ} implies 
\[  e(G) \ge \lambda^2 - \frac{6t}{n} 
> \frac{n^2}{4} - 60n.  \]
We claim that $G$ is not $120n$-far from being bipartite.  
Otherwise, if $G$ is $120n$-far from being bipartite, then Theorem \ref{thm-far-bipartite} 
implies that $G$ has at least $\frac{n}{6}(\frac{n^2}{4} - 60n 
+ 120n - \frac{n^2}{4}) = 10n^2$ triangles, a contradiction. 
Therefore, $G$ is not $120n$-far from being bipartite. 
Thus, there exists a partition $V(G)=S\cup T$ of 
the vertices of $G$ such that    
\begin{equation*}  \label{eq-EST}
e(S) + e(T) < 120n. 
\end{equation*}
 Consequently, we have 
\[ e(S,T)> e(G) - 120n  
> \frac{n^2}{4} - 180n. \] 
Without loss of generality, we may assume that 
$1\le |S| \le |T|$. Suppose on the contrary that 
if $|S| \le \frac{n}{2} - 14\sqrt{n}$, then by $|S| + |T|=n$, we have 
$|T| \ge \frac{n}{2}  + 14\sqrt{n}$. It follows that 
\[ e(S,T)\le |S| |T| \le 
\left(\frac{n}{2} - 14\sqrt{n} \right)
\left(\frac{n}{2} + 14\sqrt{n} \right) 
< \frac{n^2}{4} - 196n, \] 
which is a contradiction. Thus, we obtain $|S| > 
\frac{n}{2} - 14\sqrt{n}$ and 
$|T| < \frac{n}{2} + 14\sqrt{n}$. 
        \end{proof}

Throughout the paper, we always assume that 
$V(G)=S\cup T$ is a partition with a maximum cut, i.e., the bipartite subgraph $G[S,T]$ has the maximum number of edges. Lemma \ref{lem-partition} guarantees that there exists such a partition with    $e(S,T) >\frac{1}{4} n^2 - 180n$ and 
        $\frac{n}{2} - 14\sqrt{n} < |S|, |T| < 
\frac{n}{2} + 14\sqrt{n}$.  
Unlike the argument in Section \ref{sec4}, 
the above bipartition is not enough for our purpose. 
Under this circumstance, the graph $G$ has many triangles, and 
applying Theorem \ref{thm-far-bipartite}  can only yield an  
approximate vertex partition. After the bipartition is obtained,  
the remaining task is to perform a more careful structural analysis. 
We need to show that the extremal graph is an almost-balanced bipartite graph.

In the sequel, 
we define two sets of ‘bad’ vertices of $G$. Namely, 
\[ L:= \left\{v\in V(G): d(v) \leq 
\left(\frac{1}{2}-\frac{1}{200} \right) n\right\} .\]  
        For a vertex $v\in V(G)$, let $d_S(v) = |N(v) \cap S|$ and $d_T(v) = |N(v) \cap T|$. We denote 
\[ W: = \left\{ v\in S: d_S(v) \geq  \frac{n}{150} \right\} \cup 
\left\{v \in T: d_T(v) \geq  \frac{n}{150} \right\}. \]
In order to show that 
$G$ is the expected extremal graph 
$K_{\lceil \frac{n}{2} \rceil, \lfloor \frac{n}{2} \rfloor}^{+2}$, 
we need to prove that 
the bad sets $L$ and $W$ are empty. 
Next, we show that $W$ and $L$ are small sets. 

\begin{lemma}\label{Lupper}  
We have 
$$|L|< \frac{n}{403}.$$
\end{lemma}
\begin{proof}
 Suppose that $|L| \ge \frac{n}{403}$.
 Then let $L^{\prime}\subseteq L$ with $|L^{\prime}|= \frac{n}{403}$.
It follows that
   \begin{eqnarray*}
e(G\setminus L^{\prime}) &\ge& e(G)-\sum_{v\in L^{\prime}}d(v)\\
&\ge& \frac{n^2}{4} - 60n 
-  \frac{n}{403}\left(\frac{1}{2}-\frac{1}{200} \right)n\\
&>& \frac{1}{4} \left(n- \frac{n}{403} \right)^2 + 4 \times 10^{-6} n^2, 
     \end{eqnarray*}
where the last  inequality holds for $n\ge 8.8\times 10^{6}$. 
By Corollary~\ref{cor-MM}, the subgraph $G\setminus L^{\prime}$ contains more than 
$\frac{4}{3} \times 10^{-6}n^3$ triangles, which contradicts with the fact $t(G) < 10n^2$ for every $n\ge 7.5 \times 10^{6}$. 
So we have $|L| < \frac{n}{403}$, as needed. 
 \end{proof}

   In the proof of Lemma \ref{Lupper}, one can apply  
 Theorem \ref{thm-KMP} to obtain a better bound on $|L|$, for example, $|L| \le 150$. However,  Theorem \ref{thm-KMP}  requires $n\ge 1/ \delta$, where $\delta >0$ is an unspecified constant. To get an exact bound on $n$, rather than on $n$ being sufficiently large, we need to use Corollary~\ref{cor-MM}, as a slightly worse bound on $|L|$ is allowed for our purpose.

\begin{lemma}\label{W-Lenpty}
We have
$$|W| < 36000.$$
\end{lemma}

\begin{proof}
We denote $W_1=W\cap S$ and $W_2=W\cap T$. Then 
 $$2e(S) =\sum_{u\in S}d_{S}(u) \ge  \sum_{u\in W_1}d_S(u)\ge  \frac{n}{150} |W_1| $$ 
 and 
 \[  2e(T) = \sum_{u\in T}d_{T}(u)\ge  \sum_{u\in W_2}d_T(u)
 \ge \frac{n}{150} |W_2| . \]
  So we obtain 
  \begin{equation*} \label{WL-2}
  e(S)+e(T)\ge (|W_1|+|W_2|)\frac{ n}{300} 
  =\frac{|W|  n}{300}.
  \end{equation*}
 On the other hand, 
 Lemma \ref{lem-partition} implies that 
 \[  e(S)+e(T)< 120n. \] 
  Then we get $\frac{|W| n}{300}< 120n$, 
that is,
$ |W|< 36000$, as desired. 
  \end{proof}

  We also need the following  inclusion-exclusion principle.

\begin{lemma}\label{set}
Let $A_1, A_2,\dots, A_k$ be $k$  finite sets. Then
\begin{equation*}
\left| \bigcap_{i=1}^k A_i \right| \ge \sum_{i=1}^k |A_i|-(k-1)\left|\bigcup_{i=1}^k A_i\right|.
\end{equation*}
\end{lemma}

  \begin{lemma} \label{lem-W-sub-L}
  We have $W \subseteq L$. 
  \end{lemma}
  
  \begin{proof}
We will show that if $u\in V(G)$ and $u\notin L$, then $u\notin W$. 
For notational convenience, we denote $L_1=L\cap S$ and $L_2=L\cap T$. Without loss of generality, 
we may assume that $u\in S$ and 
   $u\notin L_1.$ Since  $S$ and $T$ form a maximum cut, 
  we can get  $d_T(u)\ge \frac{1}{2}d(u)$. Since $u\not\in L_1$, we have $d(u)> \left(\frac{1}{2}-\frac{1}{200} \right)n$. So
   $$d_T(u)\ge \frac{1}{2}d(u) > 
   \left(\frac{1}{4}-\frac{1}{400} \right)n.$$
       Recall that 
  $|L|\le \frac{n}{403}$, $|W| \le 36000$ and 
  $|S|> \frac{n}{2} - 14\sqrt{n} $,   
we have 
$  |S\setminus (W\cup L)| > \frac{n}{3}$. 

  We claim that $u$ has at most one neighbor in 
  $S\setminus (W\cup L)$. 
  Indeed, suppose on the contrary that 
   $u$ is adjacent to two vertices  $u_1, u_2$ in $S\setminus (W\cup L)$.  
   For $i\in \{1,2\}$, since $u_i\not\in L$,  we have $d(u_i)> \left(\frac{1}{2}-\frac{1}{200}\right )n$. At the same time, since $u_i\notin W$, 
   we have $d_S(u_i)< \frac{n}{150}$. So
  \[  d_T(u_i)=d(u_i)-d_S(u_i)> \left(\frac{1}{2}-\frac{1}{200} - \frac{1}{150} \right)n. \] 
   By Lemma~\ref{set}, for each $i \in \{1,2\}$, 
   we have
   \begin{eqnarray*}
 |N_T(u)\cap N_T(u_i)| 
 &\ge&  |N_T(u) | + |N_T(u_i) |  -|T| \\
& \ge &\left (\frac{1}{4}-\frac{1}{400}\right)n+\left(\frac{1}{2}-\frac{1}{200} -\frac{1}{150} \right)n 
- \left( \frac{n}{2} + 14 \sqrt{n} \right)\\
 &>&\frac{n}{5}, 
\end{eqnarray*}
 where the last inequality holds for $n\ge 1.6\times 10^5$.   
Let $B_i$ be the common neighbors of $u$ 
and $u_i$ in $T$. 
Then $|B_i|>n/5$ for each $i\in \{1,2\}$. 
Observe that for any two vertices $v_1\in B_1$ and 
 $v_2\in B_2$ with $v_1\neq v_2$, 
the triangles $uu_1v_1$ and $uu_2v_2$ 
form a copy of the bowtie. 
In other words, we can find at least 
$\frac{1}{2}\frac{n}{5}(\frac{n}{5} -1) > \frac{n}{2}$ copies of the bowtie, a contradiction. 
Therefore $u$ is adjacent to at most one vertex in  $S\setminus (W\cup L)$.

Using the above claim, for $n\ge 8.7\times 10^{6}$, we have 
\[  d_S(u) \le 1+ |W| +|L| 
<  1+ 36000 + \frac{n}{403} < \frac{n}{150}. \] 
By definition, we get $u\notin W$.
This completes the proof. 
 \end{proof}

\begin{lemma}  \label{lem-one-edge}
Both $G[S\setminus L]$ and $G[T\setminus L] $ 
 contain neither three pairwise disjoint edges nor two intersecting edges.   
Furthermore, 
we have $e(S\setminus L) \le 2$ and 
$e(T\setminus L) \le 2$. So 
there exist independent sets
 $I_S\subseteq S\setminus L$  and $I_T\subseteq T\setminus L$ such that
$ |I_S|>  |S|- \frac{n}{402} $ and $|I_T|> |T|- \frac{n}{402} $. 
 \end{lemma}

\begin{proof}
We proceed the proof in two cases. 

{\bf Case 1.}  
If $e_1,e_2$ and $e_3$ are pairwise disjoint edges in $G[S\setminus L]$, 
then we can denote 
$e_1=\{u_1,u_2\}$, $e_2=\{u_3,u_4\}$ 
and $e_3=\{u_5,u_6\}$. 
For $1\le i\le 6 $,   
since $u_i \notin L$, we get  
$d(u_i)> \left(\frac{1}{2}-\frac{1}{200} \right)n$.
By Lemma~\ref{lem-W-sub-L}, 
we have $u_i \notin W$ and 
$d_S(u_i)<  \frac{n}{150}$. Hence
$d_T(u_i)=d(u_i)-d_S(u_i)> \left(\frac{1}{2}-\frac{1}{200}-\frac{1}{150}\right)n.$
By Lemma~\ref{set}, we  get 
    \begin{eqnarray*}
\left |\bigcap_{i=1}^{6}N_T(u_i) \right| &\ge &  \sum_{i=1}^{6}|N_T(u_i)|-5 \left|\bigcup_{i=1}^{6}N_T(u_i) \right| \\
& \ge& \left (\frac{1}{2}-\frac{1}{200}-\frac{1}{150} \right)n\cdot 6 - 
5 \left(\frac{n}{2}+ 14 \sqrt{n} \right) \\
 &>& \frac{n}{6}, 
\end{eqnarray*}
where the last inequality follows by $n\ge 7.1\times 10^{4}$. 
Observe that 
each vertex of the common neighbors of 
$\{u_1,u_2,\ldots ,u_6\}$ leads to three copies of the bowtie. Consequently,  there are 
 more than $\frac{n}{2}$ copies of the bowtie in $G$, 
 which is a contradiction. 

{\bf Case 2.}  
Assume that $e_1$ and $e_2$ are intersecting edges in $G[S\setminus L]$,  
and $e_1,e_2$ form a copy of $K_{1,2}$. 
We denote 
$e_1=\{u_1,u_2\}$ and $e_2=\{u_1,u_3\}$. Similarly, we can see that 
for each $i \in \{2,3\}$, 
   we have
   \begin{eqnarray*}
 |N_T(u_1)\cap N_T(u_i)| 
 &\ge&  |N_T(u_1) | + |N_T(u_i)|  -|T| \\
& \ge & 
2 \left(\frac{1}{2}-\frac{1}{200} -\frac{1}{150} \right)n 
- \left( \frac{n}{2} +14\sqrt{n}\right)\\
 &>&\frac{n}{3}, 
\end{eqnarray*}
where the last inequality holds for $n\ge 9541$. 
In  other words, there are more than 
 $n/3$ choices of a vertex $v_2\in T$ 
such that $u_1u_2 v_2$ forms a triangle. 
Similarly, we have at least $n/3$ choices of a vertex 
$v_3\in T$ such that $v_3\neq v_2$ and 
$u_1u_3v_3$ forms a triangle. 
Thus, we can find $\frac{1}{2}\frac{n}{3} (\frac{n}{3}-1) $ copies of the bowtie in $G$, a contradiction. 
 
 To sum up, we have proved that 
both $G[S\setminus L]$ and $G[T\setminus L]$ 
have at most two non-trivial components, and each of which consists of a single edge. 
Consequently, we get $e(S\setminus L)\le 2$ and $e(T\setminus L)\le 2$. 
Now, by deleting  
 one vertex of each edge in $G[S\setminus L]$, we can obtain 
 a large independent set. 
 So there exists an independent set
 $I_S\subseteq S\setminus L$ such that
  $
  |I_S|\ge |S\setminus L| - 2 
  > |S| - \frac{n}{402}$ by Lemma \ref{Lupper}. 
 The same argument gives that there  is an independent set $I_T\subseteq T\setminus L$ with
$ |I_T|> |T|- \frac{n}{402}$.  
 \end{proof}

In the sequel, let $\mathbf{x}\in \mathbb{R}^n$ 
be a non-negative eigenvector corresponding to $\lambda (G)$. 
We write ${x}_v$ for the entry of $\mathbf{x}$ 
corresponding to a vertex $v$. 
By scaling, we may assume further  that the vector $\mathbf{x}$  has 
the maximum entry equal to $1$. 
Let $z$ be a vertex of $G$ with the maximum eigenvector entry and ${x}_z=1$. Then 
 $$ 
 \frac{n}{2} < 
 \lambda (G) = \lambda (G) {x}_z = 
 \sum_{w\in N( z)}  {x}_w \le d(z).$$
 By definition, we have $z\notin L$. 
We may assume that $z\in S\setminus L$.

\begin{lemma} \label{lem-IT}
We have $ \sum\limits_{ v\in I_T}  {x}_v >  \frac{n}{2} - \frac{n}{201}$.
\end{lemma}

\begin{proof} 
  By Lemma \ref{lem-W-sub-L}, we know that $W\subseteq L$ and 
  $|L|< \frac{n}{403}$. 
From Lemma~\ref{lem-one-edge}, we have 
$d_{S\setminus L}(z) \le 2$ and 
 $  d_S(z)\le d_{S\setminus L}(z) + |L| < \frac{n}{402}$.  
 Therefore, we get
 \begin{eqnarray*}
 \lambda (G)&=& \lambda (G) {x}_z  = 
 \sum_{v\in N_S(z)} {x}_v 
 +\sum_{v\in N_T(z)} {x}_v\\
 &=& \sum_{v\in N_S(z)} {x}_v+  \sum_{v\sim z, v\in I_T} {x}_v+\sum_{v\sim z , v\in T\setminus I_T} {x}_v\\  
 &<&  \frac{n}{402} +\sum_{ v\in I_T} {x}_v+ |T \setminus I_T|\\ 
&<&  \sum_{ v\in I_T} {x}_v  + \frac{n}{201}, 
\end{eqnarray*}
where $I_T$ is the independent set from Lemma \ref{lem-one-edge}. Recall that $\lambda (G) \ge 
\lambda (K_{\lceil \frac{n}{2} \rceil, \lfloor \frac{n}{2} \rfloor}^{+2})  > \frac{n}{2}$. 
 It follows that 
$ \sum_{ v\in I_T} {x}_v> \frac{n}{2} - \frac{n}{201}$, as required.
\end{proof}

\begin{lemma}\label{Lempty}
 We have $L = \varnothing$. 
\end{lemma}

\begin{proof}
Suppose on the contrary that there exists a vertex  $v\in L$. Then $d(v)\le (\frac{1}{2}-\frac{1}{200})n$.
We define a graph $G^+$ with vertex set $V(G)$ and edge set 
\[ E(G^+) := E(G \setminus \{v\}) \cup \{vw: w\in I_T\}. \]  
 Note that adding a vertex incident with all vertices in $I_T$ does not create any triangles 
since $I_T$ is an independent set. 
By Lemma \ref{lem-IT},  we have 
\begin{eqnarray*}
\lambda (G^+) - \lambda (G) &\geq &
 \frac{\mathbf{x}^{\top}\left(A(G^+) - A(G)\right) \mathbf{x}}{\mathbf{x}^{\top} \mathbf{x}} \\ 
 & =& \frac{2 {x}_v}{\mathbf{x}^{\top}\mathbf{x}}\left( \sum_{w\in I_T} {x}_w - \sum_{u\in N_G(v)} {x}_u\right) \\
& \geq & \frac{2 {x}_v}{\mathbf{x}^{\top}\mathbf{x}} 
\left( \frac{n}{2} - \frac{n}{201} - \Bigl(\frac{1}{2}-\frac{1}{200} \Bigr)n \right)\\
&=& \frac{2 {x}_v}{\mathbf{x}^{\top}\mathbf{x}} \left( \frac{n}{200} - \frac{n}{201}\right)> 0,
\end{eqnarray*}
which contradicts 
with the maximality of $\lambda(G)$. 
So $L$ must be empty. 
\end{proof}

By Lemmas \ref{lem-one-edge} and \ref{Lempty},  
we  get $e(S) \le 2$ and $e(T) \le 2$. 

\begin{lemma} \label{STle3}
We have $e(S) + e(T) \le 3$. 
\end{lemma}

\begin{proof}
Note that $L=\varnothing$ by  Lemma \ref{Lempty}. 
Then for every vertex $v\in S$, we have 
$d(v) > (\frac{1}{2} - \frac{1}{200})n $ and 
$d_S(v) < \frac{n}{150}$. So $d_T(v) > (\frac{1}{2} - \frac{1}{200} - \frac{1}{150})n$. The corresponding degree condition also holds for 
each vertex of $T$. Suppose on the contrary 
that $e(S)=2 $ and $ e(T)=2$. 
Let $e_1=\{v_1,v_2\}$ and 
$e_2=\{v_3,v_4\}$ be two edges in $G[S]$. 
Hence, by Lemma~\ref{set}, we get that 
\[  \left| \bigcap_{i=1}^4 N_T(v_i) \right| >  4\left(\frac{1}{2} - \frac{1}{200} - \frac{1}{150} \right)n - 3\left(\frac{n}{2} + 14\sqrt{n} \right) > \frac{2n}{5},\]
where the last inequality holds for $n\ge 6.3\times 10^5$. 
Each vertex of the common neighbors of $v_1,v_2,v_3,v_4$ in $T$ can yield 
a copy of the bowtie. 
So there are more than $\frac{2n}{5}$ bowties 
with the center in $T$. The similar result holds 
for the edges in $T$.  
So there are more than $\frac{4n}{5}$ copies of bowtie in 
$G$, a contradiction. Therefore, we have 
$e(S) + e(T)\le 3$. 
\end{proof}

\subsection{Refining the structure} 

\label{sec-5-3}

In the proof of Lemma \ref{lem-partition}, 
we obtained that $G$ has 
less than $10n^2$ triangles and $\lambda (G) > {n}/{2}$,  which together with 
Lemma \ref{thm-BN-CFTZ-NZ} and Theorem \ref{thm-far-bipartite} yields an approximate 
structure of  $G$. 
Fortunately, using Lemma \ref{STle3}, 
we get $e(S)+e(T)\le 3$, which implies that 
$G$ has at most 
$3({n}/{2} + 14\sqrt{n})$ triangles. 
Next, we refine the bipartition of 
$G$ by using Lemma \ref{thm-BN-CFTZ-NZ} again, 
and then we show that 
$G$ is an almost-balanced complete bipartite graph.

\begin{lemma}\label{STlambdarefine}
We have
\begin{equation*}
e(G)\ge \left\lfloor \frac{n^2}{4} \right\rfloor - 5,
\end{equation*}
\[ e(S,T) \ge \left\lfloor \frac{n^2}{4} \right\rfloor - 8  \]
and 
\begin{equation*}
\frac{n}{2}-2
\le |S|,  |T|\le \frac{n}{2}+2. 
\end{equation*}
\end{lemma}

\begin{proof}
 From Lemma \ref{STle3}, 
  we have $e(S) + e(T) \leq 3$.
  Since any triangle of $G$ must contain an 
  edge with two endpoints in $S$ or $T$, 
  the number of triangles in $G$ is bounded above by 
  $3({n}/{2} +14 \sqrt{n})$. 
  Moreover, we get from Lemma \ref{lem-F-G} that 
  \[ \lambda^2(G) \ge \lambda^2 (K_{\lceil \frac{n}{2} \rceil, \lfloor \frac{n}{2} \rfloor}^{+2}) > 
  \left\lfloor \frac{n^2}{4} \right\rfloor  +4. \]
       By Lemma \ref{thm-BN-CFTZ-NZ}, 
       we have $$e(G) > \lambda^2-\frac{6t}{n} 
       > \left\lfloor \frac{n^2}{4} \right\rfloor  - 6.$$
Then 
\[ e(S,T) = e(G) - e(S) - e(T)
 > \left\lfloor\frac{n^2}{4} \right\rfloor -9. \]
Note that $|S| + |T|=n$. The AM-GM inequality gives 
$ |S| |T| \le \lfloor n^2/4 \rfloor $. 
So $|S||T| - e(S,T) \le 
\lfloor n^2/4  \rfloor - e(S,T) <9$. 
In other words, there are at most $8$ edges missed between $S$ and $T$. 
Without loss of generality, 
we may assume that $|S| \le |T|$. 
       Suppose on the contrary  
       that $|S|\le \frac{n}{2}-3,$ then $|T|=n-|S|\ge \frac{n}{2}+3$. It follows that 
       $e(S,T)\le  |S||T|\le  \left (\frac{n}{2}-3 \right) \left(\frac{n}{2}+3 \right)= \frac{n^2}{4}-9$, 
        which is a contradiction. 
        Thus, we have
        $ \frac{n}{2}-2
\le |S|,  |T|\le \frac{n}{2}+2. $
\end{proof}

\begin{lemma} \label{lem-EST2}
We have $e(S) + e(T) = 2$. 
\end{lemma}

\begin{proof}
From Lemma \ref{STle3}, 
we obtain that $e(S) \le 2,e(T)\le 2$ and 
$e(S) +e(T) \le 3$. 

For convenience, we denote 
$|S|=s$ and $|T|=t$. 
Firstly, if $e(S) + e(T)=0$, 
then $G$ is a bipartite graph with color classes $S$ and $T$. 
Then we have $\lambda (G)\le \sqrt{st} 
\le \lambda (T_{n,2})$, which is a contradiction since  
$\lambda (G) \ge 
\lambda (K_{\lceil \frac{n}{2} \rceil, \lfloor \frac{n}{2} \rfloor}^{+2}) > \lambda (T_{n,2})$. 

Secondly, we suppose on the contrary that 
 $e(S)+e(T)=1$, then we may assume that  
$e(S)=1$ and $ e(T)=0$. 
By Lemma \ref{STlambdarefine}, 
we have $\frac{n}{2} -2 \le |S| \le \frac{n}{2} +2$.  
Let $K_{s,t}^+$ be the graph obtained from 
$K_{s,t}$ by adding an edge to the part of size $s$. 
So $G$ is a subgraph of $K_{s,t}^+$ and 
$G$ does not contain a copy of the bowtie. 
By calculation, we know that 
$\lambda (K_{s,t}^+)$ is the largest root of 
\[  h_{s,t}(x) := x^3 - x^2 -s t x + s t - 2 t . \]
For each $\frac{n}{2} -2 \le s\le \frac{n}{2} +2$, 
we will prove that 
 \[  \lambda (K_{s, t}^{+}) < 
\lambda (K_{\lceil \frac{n}{2} \rceil, \lfloor \frac{n}{2} \rfloor}^{+2}), \]
and get a contradiction.   
We prove this inequality in two cases. 
For even $n$, 
recall in Lemma  \ref{lem-F-G} that $\lambda 
(K_{\frac{n}{2},\frac{n}{2}}^{+2})$ 
is the largest root of $f(x)$. 
Clearly, we have 
\[  h_{s,t}(x) - f(x)= 
(n^2/4 - st)(x-1)+2(n-t).  \]
It is easy to check that for each $s\in[ \frac{n}{2} -2, \frac{n}{2} +2]$ and $t=n-s$, we have 
$h_{s,t}(x) - f(x) >0$ for every $x> {2}$.  
Then Lemma \ref{lem-roots}  
implies $\lambda (K_{s,t}^+) < \lambda 
(K_{\frac{n}{2},\frac{n}{2}}^{+2})$, 
a contradiction. 
For odd $n$, let $g(x)$ be defined in Lemma \ref{lem-F-G}. Then 
\[  h_{s,t}(x) - g(x) = ((n^2-1)/4 - st)(x-1) + 
2(n-t) - 3/2.    \]
For each $s\in [\frac{n+1}{2} - 2, \frac{n-1}{2} +2]$, 
it follows that $h_{s,t}(x) - g(x) >0$ for any $x>0$. 
Consequently, we have $\lambda (K_{s,t}^+) < \lambda (K_{ \frac{n+1}{2}, \frac{n-1}{2}}^{+2})$, 
which is a contradiction. 

Thirdly,  
we shall prove that $e(S) + e(T)\neq 3$. 
Without loss of generality, we may assume 
on the contrary  that 
$e(S)=2$ and $e(T)=1$. 
In this case, we will show that 
$G$ is a subgraph of $G_1$ in Figure \ref{fig-S2T1}.  
Let $u_1u_2$ and $u_3u_4$ be two edges in 
$G[S]$, and let $v_1v_2$ be the unique edge in $G[T]$. 
Using Lemma \ref{STlambdarefine}, 
we have $e(S,T) \ge \lfloor n^2/4\rfloor -8$ and so 
$|S||T| - e(S,T) \le \lfloor n^2/4\rfloor - e(S,T)\le 8$. 
Thus, we lose at most $8$ edges from the complete bipartite graph between $S$ and $T$. 
Consequently, there are at least $t-8$ (blue) bowties 
that center at a vertex of $T$ and contain the edges $u_1u_2$ and $u_3u_4$. 
Similarly, there are at least $s-8$ (green) triangles containing 
the edge $v_1v_2$. 
We claim that $v_1$ has at most one neighbor in 
$\{u_1,u_2\}$. Indeed, if $v_1u_1\in E(G)$ and 
$v_1u_2\in E(G)$, then $v_1u_1u_2$ forms a triangle. Combining with those $s-8$ 
triangles mentioned in the above, 
we can find $s-8$ new bowties centered at $v_1$, 
which yields at least $s+t-16 =n-16$ bowties in $G$, 
a contradiction.  Similarly, $v_1$ has at most one neighbor in $\{u_3,u_4\}$. Moreover, the same result holds for the vertex $v_2$. Without loss of generality, 
we may assume that $v_1u_1 \in E(G)$ and $v_1u_3 \in E(G)$. Under this assumption, 
we claim that $v_2$ can not be adjacent to 
$u_1$ or $u_3$. Otherwise, if $v_2u_1\in E(G)$, 
then it leads to at least $t-10$ new bowties centered at $u_1$ and so $G $ has at least $2t- 18 \ge n-22$ 
bowties in total, a contradiction. 
Thus, $v_2$ is not adjacent to $u_1$ or $u_3$. 
Note that $v_2$ can be adjacent to $u_2$ and $u_4$. 
We conclude that 
$G$ is a subgraph of $G_1$.

 \begin{figure}[H]
\centering
\includegraphics[scale=0.85]{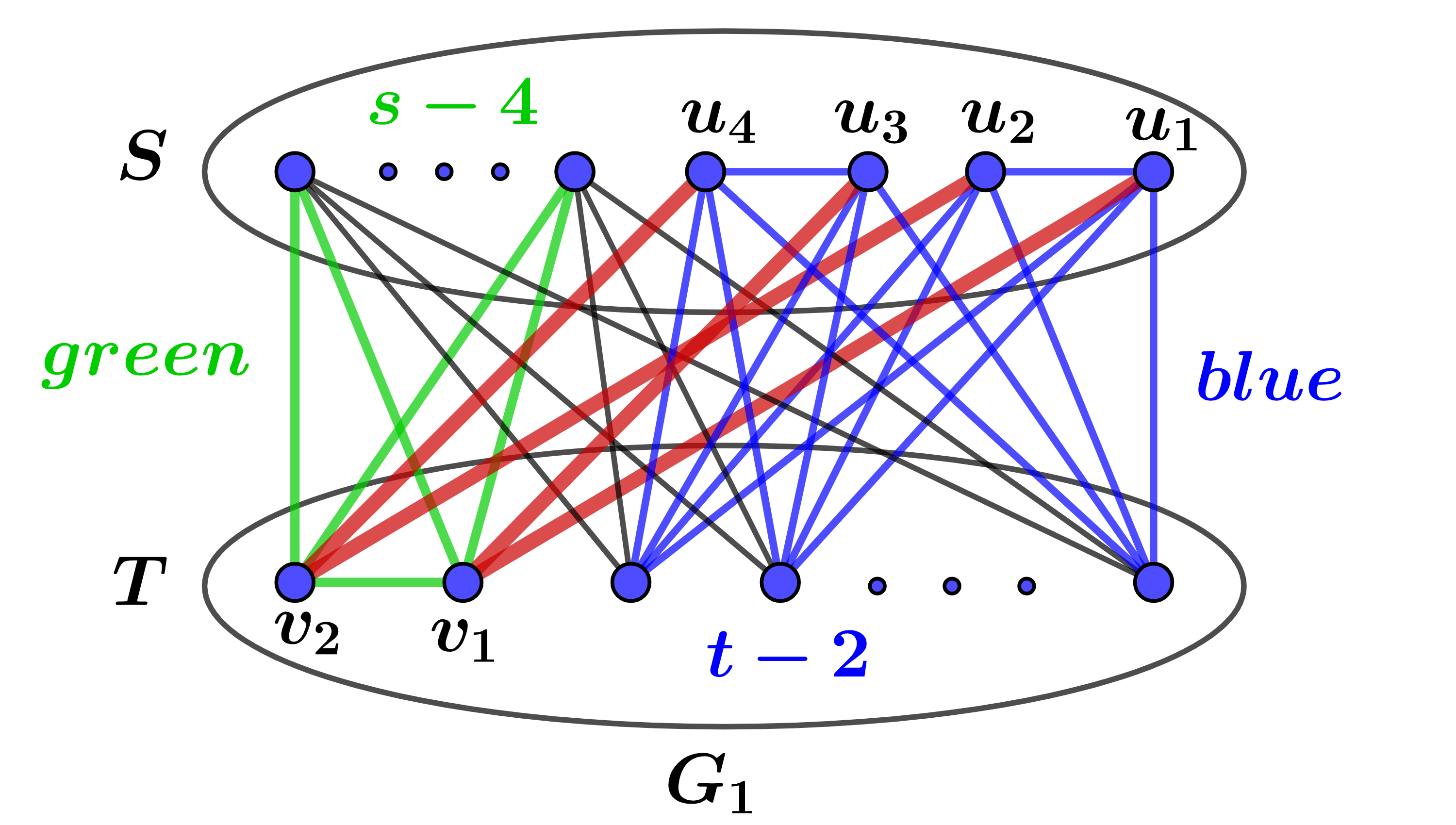} 
\caption{The graph $G_1$.}
\label{fig-S2T1}
\end{figure}
 
 Observe that the graph $G_1$ in above contains $t-2$ bowties, 
 and $G_1$ misses $4$ edges between $S$ and $T$. 
 So we get $e(G_1)=e_{{\blue G_1}}(S,T) +3 \le \lfloor n^2/4\rfloor -1$.  
Next, we will deduce a contradiction by showing 
\[ \lambda (G_1) < 
\lambda (K_{\lceil \frac{n}{2} \rceil,\lfloor \frac{n}{2} \rfloor}^{+2}). \]  
By computation, we obtain that 
$\lambda (G_1)$ is the largest root of 
\[  f_{s,t}(x) := x^4  - 2 x^3 + 7 x^2 -s t x^2 
 + 2 s t x - 2 s x  - 4 t x - 4 t  - 2 s + s t  + 8. \]
 For notational convenience, 
we partition the proof 
according to the parity of $n$. 
Assume that $n$ is even. 
For each $s\in [\frac{n}{2} -2 ,\frac{n}{2} +2]$ and $t=n-s$, 
it is easy to verify that 
$f_{s,t}(\frac{n}{2}) >0$. 
Indeed, for example, if $s= \frac{n}{2}$ and 
$t=\frac{n}{2} $, then 
$ f_{s,t}(x) = 8 + 7 x^2 - 2 x^3 + x^4 - 3 n (1 + x) 
- \frac{1}{4} n^2 (-1 - 2 x + x^2)$. 
So $f_{s,t}(\frac{n}{2}) = \frac{1}{2} n^2 - 3 n +8>0$.  
By Lemma \ref{lem-root}, we get 
$  \lambda (G_1)< \frac{n}{2} < 
\lambda (K_{\frac{n}{2}, \frac{n}{2}}^{+2})$.  
 Assume that $n$ is odd. 
In this case, we have
 $s\in [\frac{n+1}{2} -2,  \frac{n+1}{2} +1]$. 
A similar calculation implies that 
$\lambda (G_1)< \frac{n}{2} < \lambda (K_{\frac{n+1}{2}, \frac{n-1}{2}}^{+2})$. 
\end{proof}

By Lemma \ref{lem-EST2},  
we can show the following 
improvement on Lemma \ref{STlambdarefine}. 

\begin{lemma}\label{ST-refine}
We have
$e(G)\ge \lfloor {n^2}/{4} \rfloor - 2$ 
and 
$e(S,T) \ge \lfloor {n^2}/{4} \rfloor - 4$. 
\end{lemma}

\begin{proof}
From Lemma \ref{lem-EST2}, 
we obtain $e(S) + e(T)=2$. 
Observe that each triangle in $G$ contains an edge within $S$ or $T$. 
So there are at most $2(\frac{n}{2} + 2)$ triangles in $G$. A similar argument in the proof of Lemma \ref{STlambdarefine} can give the desired bounds. 
\end{proof}

\begin{lemma} \label{SorT=2}
Either $e(S)=2$ or $e(T)=2$. 
\end{lemma}

\begin{proof}
By Lemma \ref{lem-EST2}, we know that 
$e(S) + e(T)=2$. So we need to exclude the case 
$e(S) =1$ and $e(T)=1$. 
Suppose on the contrary that 
$e(S)=e(T)=1$. 
Let $\{u_1,u_2\}$ and $\{v_1,v_2\}$ 
be the two edges in $S$ and $T$, respectively.  
We will present the proof in two cases. 
In each case, we will show that $G$ is a subgraph of a certain graph, which has spectral radius less than 
that of $K_{\lceil \frac{n}{2} \rceil,\lfloor \frac{n}{2} \rfloor}^{+2}$.  So we obtain 
$\lambda (G) < \lambda (K_{\lceil \frac{n}{2} \rceil,\lfloor \frac{n}{2} \rfloor}^{+2})$, which contradicts with our assumption. 

 \begin{figure}[H]
\centering
\includegraphics[scale=0.85]{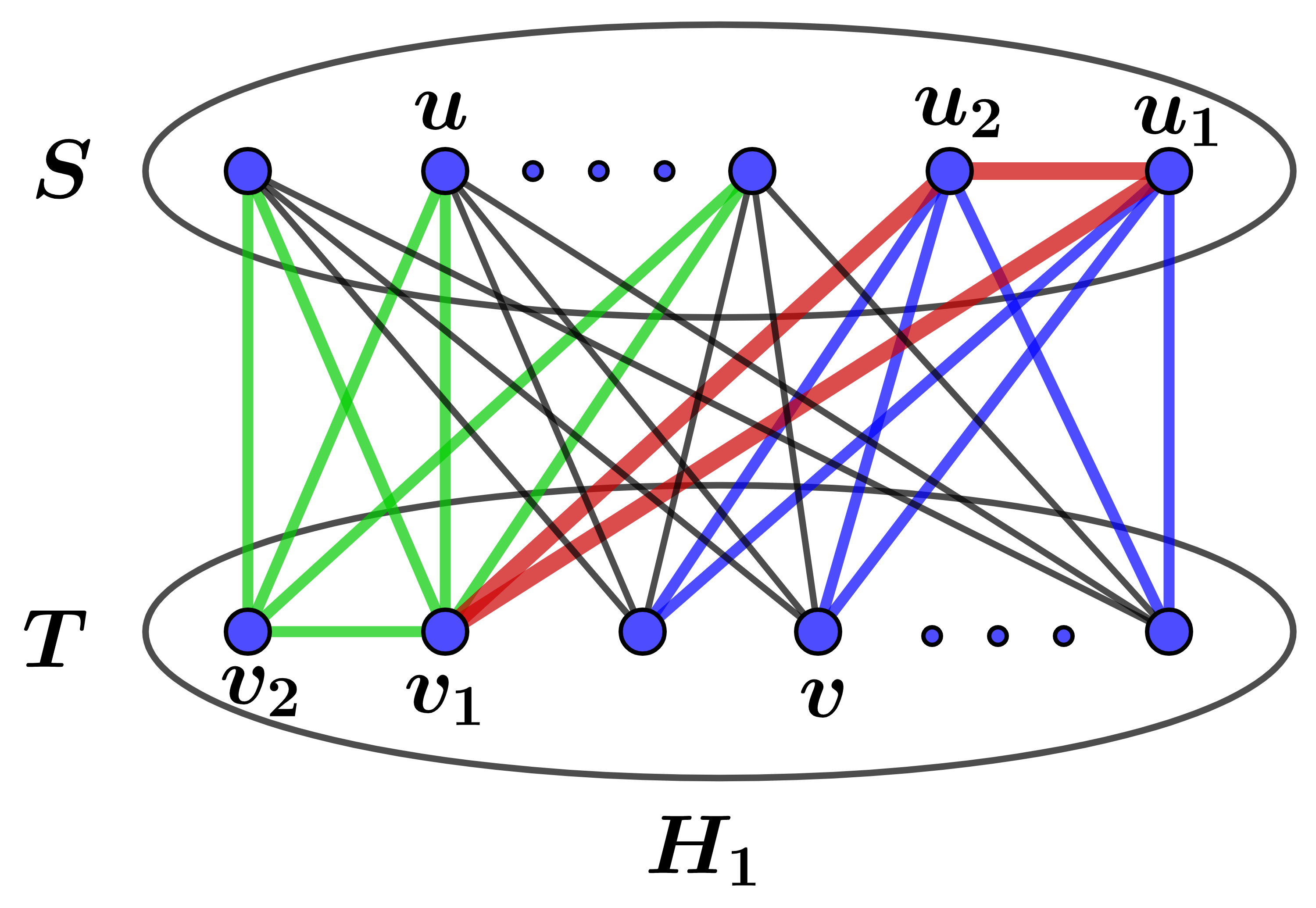} 
\quad \quad 
\includegraphics[scale=0.85]{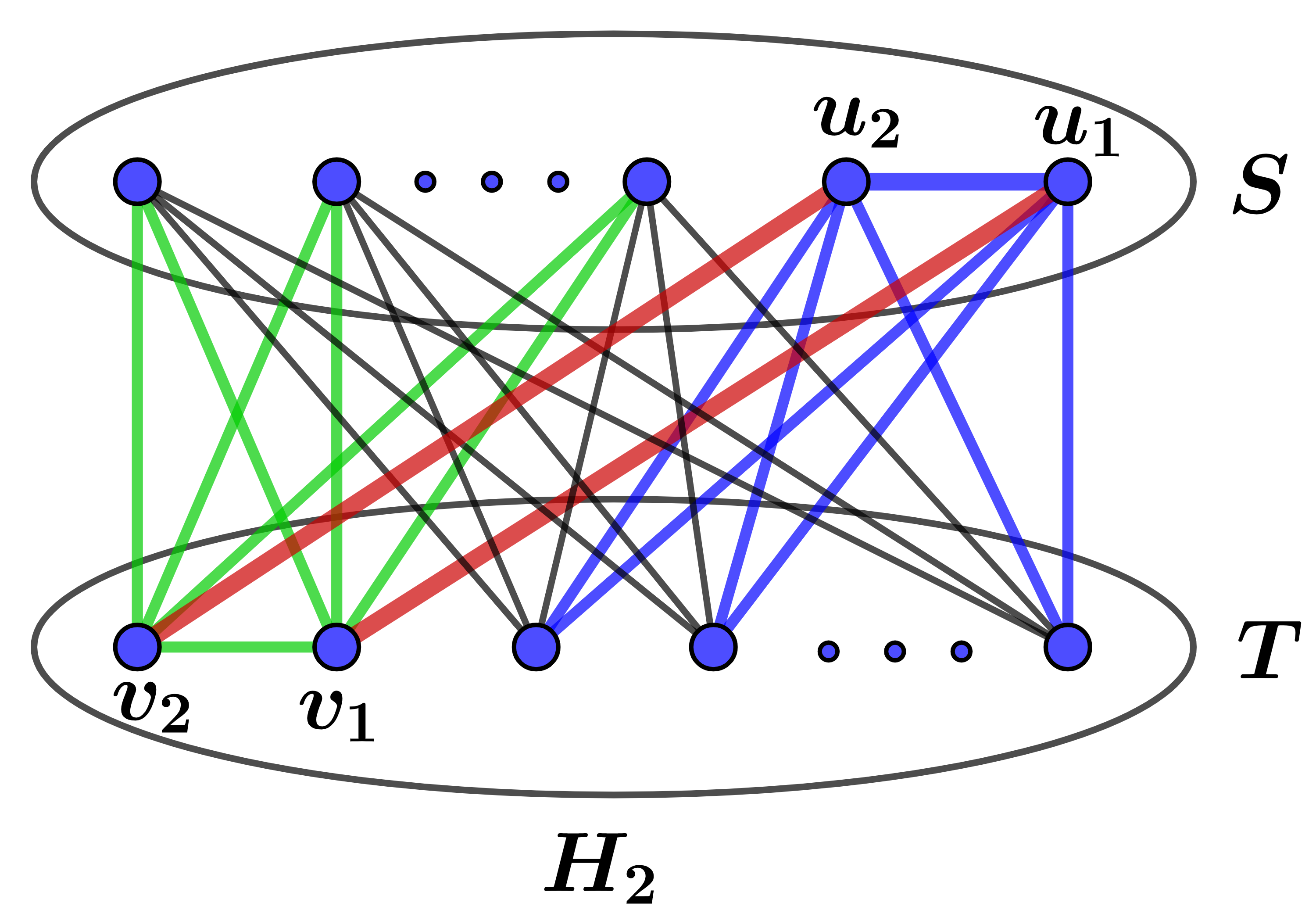}
\caption{The graphs $H_1$ and $H_2$.}
\label{fig-ST1}
\end{figure}

{\bf Case 1.} 
Suppose that 
$v_1$ is adjacent to both $u_1$ and $u_2$. 
Then each green triangle $uv_1v_2$ 
together with the red triangle $u_1u_2v_1$ 
forms a bowtie with the center $v_1$. 
Observe that $|S||T| - e(S,T) 
\le \lfloor n^2/4\rfloor - e(S,T)\le 4$, i.e., 
 at most $4$ edges are missed between 
$S$ and $T$. 
So there are at least $s-6$ green triangles and 
then there are at least $s-6$ bowties with the center 
$v_1$. 
We claim that $v_2u_1\notin E(G)$ and 
$v_2u_2\notin E(G)$. 
Otherwise, if $v_2u_1\in E(G)$, then 
$u_1v_1v_2$ forms a triangle. 
Observe that 
each blue triangle $u_1u_2v$ together with 
$u_1v_1v_2$ forms a bowtie centered at $u_1$. 
So we can find at least $t-6$ new bowties. 
 In total, we can find at least $s+t-12=n-12$ 
 bowties in $G$, a contradiction. So we must have 
 $v_2u_1\notin E(G)$. Similarly, we can get 
 $v_2u_2\notin E(G)$ as well.  
 Therefore, we conclude that 
 $G$ is a subgraph of $H_1$ in Figure \ref{fig-ST1}. 
 To deduce a contradiction, we will show that 
 \[ \lambda (H_1) 
 < \lambda (K_{\lceil \frac{n}{2} \rceil,\lfloor \frac{n}{2} \rfloor}^{+2}). \] 
  Upon computation, 
 we get that $\lambda (H_1)$ is the largest root of
 \begin{eqnarray*}
  b_{s,t}(x)&:=& x^5 - x^4 + x^3 - s t x^3 + 5 x^2 - 2 s x^2 -  2 t x^2 + s t x^2 \\
  & & 
 - 2 s x - 4 t x + 3 s t x -4 + 2 s + 2 t - s t .  
 \end{eqnarray*}
 By a direct calculation, we can verify that for 
 each $s\in [\frac{n}{2} -2, \frac{n}{2} +2]$ and $t=n-s$,  it follows that $b_{s,t}(x) > x^2f(x)$ for every 
 $x\ge 1$, 
 where $f(x)$ is the polynomial 
 in Lemma \ref{lem-F-G}. 
 For example, for even $n$, taking $s=\frac{n}{2}$ and $t=\frac{n}{2}$, we have $b_{s,t}(x)=-4 + 2 n - n^2/4 - 3 n x + (3 n^2 x)/4 + 5 x^2 - 2 n x^2 + (
 n^2 x^2)/4 + x^3 - (n^2 x^3)/4 - x^4 + x^5$. 
Then 
 \[  x^2f(x) - b_{s,t}(x) = 4 - 5 x^2 - x^3 + n (-2 + 3 x) - \tfrac{1}{4} n^2 (-1 + 3 x). \] 
 One can check that $x^2f(x) - b_{s,t}(x) < 0$ 
 for any $x\ge 1$. 
 Recall that in Lemma \ref{lem-F-G} that 
 $\lambda (K_{ \frac{n}{2},\frac{n}{2}}^{+2})$ is the largest root of $f(x)$. 
 So we get $b_{s,t}(x) > 
x^2f(x)\ge 0$ 
 for any $x\ge 1$. 
Lemma \ref{lem-roots} yields $\lambda (H_1) < 
 \lambda (K_{ \frac{n}{2},\frac{n}{2}}^{+2})$. 
 In fact, for the other case, namely, $s=\frac{n}{2}\pm 1$ or $\frac{n}{2}\pm 2$, we can show by Lemma \ref{lem-root}  that $\lambda (H_1)< \frac{n}{2} < \lambda (K_{ \frac{n}{2},\frac{n}{2}}^{+2})$, as needed. 
 
 {\bf Case 2.} 
 Suppose that $v_1$ is adjacent to $u_1$,  
  but $v_1$ is not adjacent to $u_2$. 
  In this case, the vertex $v_2$ can be adjacent to 
  $u_2$.   If $v_2$ is adjacent to $u_1$, then 
  we can turn to the previous case by symmetry. So we may assume that $v_2u_1\notin E(G)$ and 
  then $G$ is a subgraph of $H_2$ in Figure \ref{fig-ST1}. We remark here that the graph $H_2$ contains no copy of the bowtie. By computation, we know that 
  $\lambda (H_2)$ is the largest root of 
\[ d_{s,t}(x) := x^4 - 2 x^3+ 4 x^2 - s t x^2
-2 s x - 2 t x + 2 s t x.  \]  
Similarly, we can show that for every 
$s\in [\frac{n}{2} -2, \frac{n}{2} +2]$, we have 
 \[ \lambda (H_2) 
 < \lambda (K_{\lceil \frac{n}{2} \rceil,\lfloor \frac{n}{2} \rfloor}^{+2}). \] 
Next, we only prove the inequality for even $n$, 
since the case for odd $n$ can be proved similarly. 
 For $s=t=\frac{n}{2}$, we see that 
$H_2$ is a regular graph with degree $\frac{n}{2}$, 
so $\lambda (H_2)= \frac{n}{2} < \lambda (K_{ \frac{n}{2},\frac{n}{2}}^{+2})$. 
  For $s=\frac{n}{2} +1$ and $t=\frac{n}{2} -1$, 
  we can check that $d_{s,t}(\frac{n}{2}) = -n + n^2/4 >0$. 
  So Lemma \ref{lem-root} gives 
   $\lambda (H_2) < \frac{n}{2} < \lambda (K_{ \frac{n}{2},\frac{n}{2}}^{+2})$. 
  By the symmetry, the same result holds for 
  $s=\frac{n}{2}-1$ and $t=\frac{n}{2} +1$. 
  In addition, for $s=\frac{n}{2} \pm 2$ and 
  $t=n-s$, we have $d_{s,t}(\frac{n}{2} ) = -4 n + n^2 >0$. 
  To sum up, for each $s\in [\frac{n}{2}-2, \frac{n}{2} +2]$,   we have $\lambda (H_2) < \frac{n}{2} < \lambda (K_{ \frac{n}{2},\frac{n}{2}}^{+2})$, as desired. 
\end{proof}

\subsection{Completing the proof}

Having finished all the necessary preparations in Lemma \ref{lem-AS} -- 
Lemma \ref{SorT=2}, 
we are now in a position to complete the proof of our main result. 

\begin{proof}[{\bf Proof of Theorem \ref{thm-bowtie}}] 
Recall 
that  $G$ is a graph 
$\lambda (G) \ge \lambda (K_{\lceil \frac{n}{2} \rceil, \lfloor \frac{n}{2} \rfloor}^{+2})$ and $G$ 
has at most $\lfloor \frac{n}{2}\rfloor$ bowties. 
We need to show that $G$ is the extremal graph $K_{\lceil \frac{n}{2} \rceil, \lfloor \frac{n}{2} \rfloor}^{+2}$. 
By Lemmas \ref{lem-one-edge} and \ref{Lempty},  
we  get $e(S) \le 2 $ and $e(T) \le 2$. Moreover $G[S]$ and $G[T]$ 
are $K_{1,2}$-free. 
From Lemma \ref{ST-refine}, 
we have $e(G)\ge \lfloor n^2/4\rfloor -2$. 
 Furthermore, 
we know that $G$ has a vertex partition $V(G)=S\cup T$ 
such that $e(S,T)\ge \lfloor n^2/4\rfloor  -4$ and 
\[  \frac{n}{2}  -2 \le |S|, |T| \le  \frac{n}{2}  +2. \]   
By Lemma \ref{SorT=2}, we have 
$e(S)=2$ or $e(T)=2$. 
By Theorem \ref{thmNZ2021}, we know that $G$ has 
at least $\lfloor n/2\rfloor -1$ triangles.  
In what follows, we shall deduce a contradiction.

\medskip 
{\bf Case 1.} Assume that $n$ is even.

{\bf Subcase 1.1.}  $|S|= \frac{n}{2} $ and $|T|=\frac{n}{2} $.  

By the symmetry, 
we may assume that $e(S)=2$. 
Then $G$ is a subgraph of $K_{\frac{n}{2} , \frac{n}{2}}^{+2}$. 
 Since $\lambda (G) \ge \lambda (K_{\frac{n}{2} , \frac{n}{2}}^{+2})$, we obtain $G=K_{\frac{n}{2} , \frac{n}{2}}^{+2}$, 
 which is the extremal graph.

{\bf Subcase 1.2.}  $|S|= \frac{n}{2} -1$ and $|T|=\frac{n}{2} +1$.

In this case, 
 we have $n^2/4 -4 \le e(S,T) \le n^2/4 -1$. 
So there are at most three missing edges between 
$S$ and $T$. 
If $e(S)=2$, then 
$G$ is a subgraph of $K_{\frac{n}{2} -1, \frac{n}{2} +1}^{+2}$. 
Upon computation, we obtain that 
 $\lambda (K_{\frac{n}{2} -1, \frac{n}{2} +1}^{+2})$ 
is the largest root of 
\[  f_1(x):=-5 - 2 n + n^2/4 + x - (n^2 x)/4 - x^2 + x^3. \]
Recall in Lemma \ref{lem-F-G} that 
$ \lambda (K_{\frac{n}{2} , \frac{n}{2} }^{+2})$ is the largest root of 
\[  f(x)= x^3 - x^2 - (n^2 x)/4 + n^2/4 - 2 n. \]  
Observe that $f_1(x) - f(x)= -5+x >0$ for any $x>5$. 
Then $f_1(x) > f(x) \ge 0$ 
for any $x> 5$. 
Lemma \ref{lem-roots} gives $\lambda (K_{\frac{n}{2} -1, \frac{n}{2} +1}^{+2}) 
< \lambda (K_{\frac{n}{2} , \frac{n}{2} }^{+2})$, a contradiction.

If $e(T)=2$, then $G$ is a subgraph of $K_{\frac{n}{2} +1, \frac{n}{2} -1}^{+2}$. 
By computation, we know that 
$\lambda (K_{\frac{n}{2} +1, \frac{n}{2} -1}^{+2})$ 
is the largest root of 
\[  f_2(x) := 3 - 2 n + n^2/4 + x - (n^2 x)/4 - x^2 + x^3.  \]
It is easy to see that $f_2(x) - f(x)= 3+x >0 $ for any $x>0$. 
Similarly, by Lemma \ref{lem-roots}, 
we can obtain $\lambda (K_{\frac{n}{2} +1, \frac{n}{2} -1}^{+2}) < \lambda (K_{\frac{n}{2} , \frac{n}{2} }^{+2})$, 
 which is a contradiction.

{\bf Subcase 1.3.}  $|S|= \frac{n}{2} -2$ and $|T|=\frac{n}{2} +2$.

In this case, we have $e(S,T)\le |S| |T| \le 
 n^2/4  -4$, and so we must have 
$e(S,T) = n^2/4  -4$. In other words, 
the induced subgraph $G[S,T]$ is a complete bipartite graph. 
 If $e(S)=2$, then $G=K_{\frac{n}{2} -2, \frac{n}{2} +2}^{+2}$ 
 and $G$ has exactly $\frac{n}{2} +2$ bowties, 
a contradiction.  

If $e(T)=2$, then  $G= K_{\frac{n}{2} +2, \frac{n}{2} -2}^{+2}$. 
By calculation, we obtain that $\lambda (K_{\frac{n}{2} +2, \frac{n}{2} -2}^{+2}) $ is the largest root of 
\[   f_3(x) := 4 - 2 n + n^2/4 + 4 x - (n^2 x)/4 - x^2 + x^3.   \]
It is easy to verify that $f_3(\frac{n}{2}) = 4 >0$. 
Using Lemma \ref{lem-root}, 
 we have 
 \[ \lambda (K_{\frac{n}{2} +2, \frac{n}{2} -2}^{+2}) < \frac{n}{2} <
 \lambda (K_{\frac{n}{2},\frac{n}{2}}^{+2}), \]  
 which is a contradiction.

\medskip 
{\bf Case 2.} Suppose that $n$ is odd. 

{\bf Subcase 2.1.}  $|S|= \frac{n-1}{2} $ and $|T|=\frac{n+1}{2}$.

If $e(T)=2$, then $G$ is a subgraph of $K_{\frac{n+1}{2} , \frac{n-1}{2} }^{+2}$. 
Since $\lambda (G) \ge \lambda (K_{\frac{n+1}{2} , \frac{n-1}{2} }^{+2})$, we must have $G=K_{\frac{n+1}{2} , \frac{n-1}{2} }^{+2}$, as desired.   
If $e(S)=2$, then $G$ is a subgraph of $K_{\frac{n-1}{2} , \frac{n+1}{2} }^{+2}$. 
Note that $G$ has at most $\frac{n-1}{2}$ bowties and 
$K_{\frac{n-1}{2} , \frac{n+1}{2} }^{+2}$ has exactly 
$\frac{n+1}{2}$ bowties. 
So $G$ is obtained from $K_{\frac{n-1}{2} , \frac{n+1}{2} }^{+2}$ by deleting some edges 
that destroy at least one bowtie. 
Let $K_{\frac{n-1}{2} , \frac{n+1}{2} }^{+2\,\, |}$ be
the graph obtained 
from $K_{\frac{n-1}{2} , \frac{n+1}{2} }^{+2}$ by deleting 
an edge between $S$ and $T$ which is incident to an edge 
of $G[S]$. 
Furthermore, $G$ is a subgraph of 
$K_{\frac{n-1}{2} , \frac{n+1}{2} }^{+2\,\, |}$. 
It is not hard to deduce that 
$\lambda (K_{\frac{n-1}{2} , \frac{n+1}{2} }^{+2\,\, |})$ 
is the largest root of 
\begin{eqnarray*}
  f_4(x) &: =& x/2 + 4 n x - (n^2 x)/2 - (3 x^2)/2 - n x^2 + (n^2 x^2)/2 - x^3/4 \\ 
  & & -  2 n x^3 + (n^2 x^3)/4 + x^4/4 - (n^2 x^4)/4 - x^5 + x^6.  
 \end{eqnarray*}
By Lemma \ref{lem-F-G}, we know that $\lambda (K_{ \frac{n+1}{2}, \frac{n-1}{2}}^{+2})$ 
 is the largest root of 
 \[  g(x)=x^3 - x^2 + x/4 - (n^2 x)/4 + n^2/4 - 2 n + 7/4. \]    
 By computation, we get 
 \[ f_4(x) - x^3g(x) = - ({x}/{2}) (4x^2 + 
 (2n-n^2+3)x  + n^2 -8n -1). \]   
It is not hard to verify that $f_4(x) - x^3 g(x) >0$ 
for any $x\in (\frac{n}{2}, \frac{n}{2} +1)$. 
By Lemma \ref{lem-roots}, we can obtain 
$\lambda (K_{\frac{n-1}{2} , \frac{n+1}{2} }^{+2\,\, |})< 
\lambda (K_{ \frac{n+1}{2}, \frac{n-1}{2}}^{+2})$, 
a contradiction.

\medskip 
{\bf Subcase 2.2.}  $|S|= \frac{n-3}{2} $ and $|T|=\frac{n+3}{2}$. 

If $e(S)=2$, then $G$ is a subgraph of $K_{\frac{n-3}{2} , \frac{n+3}{2} }^{+2}$. 
Notice that $\lambda(K_{\frac{n-3}{2} , \frac{n+3}{2} }^{+2})$ 
is the largest root of 
\[ f_5(x) := -(33/4) - 2 n + n^2/4 + (9 x)/4 - (n^2 x)/4 - x^2 + x^3.  \]
Moreover, we have $f_5(x) - g(x)= -10 + 2 x >0$ 
for any $x>5$. Then 
$f_5(x) > g(x) \ge 0$ for any 
$x>5$. 
 So Lemma \ref{lem-roots}
 implies $\lambda (K_{\frac{n-3}{2} , \frac{n+3}{2} }^{+2}) 
< \lambda (K_{\frac{n+1}{2} , \frac{n-1}{2} }^{+2})$, 
a contradiction. 

If $e(T)=2$, then $G$ is a subgraph of $K_{\frac{n+3}{2} , \frac{n-3}{2} }^{+2}$. Then we can check that 
$\lambda (K_{\frac{n+3}{2} , \frac{n-3}{2} }^{+2})$ 
is the largest root of 
\[ f_6(x):= 15/4 - 2 n + n^2/4 + (9 x)/4 - (n^2 x)/4 - x^2 + x^3.  \]
It is easy to verify that $f_6(x) - g(x)= 2x+2 >0$ 
for any $x>0$.  
Applying Lemma \ref{lem-roots},
 we get  $ \lambda (K_{\frac{n+3}{2} , \frac{n-3}{2} }^{+2}) <   \lambda (K_{\frac{n+1}{2} , \frac{n-1}{2} }^{+2})$, a contradiction. 
\end{proof}

\section{Concluding remarks}
\label{sec7}

In this paper, we have provided a stability variant of 
a spectral counting result due to Ning and Zhai. 
Moreover, we have proved a spectral counting 
result for the bowtie, which solves an open problem 
proposed by the authors in \cite{2022LLP}.  
 Next,  we are going to conclude some spectral extremal graph problems concerning the number of 
 triangles and bowties. 
 
 \medskip 
 
An old result of Nosal \cite{Nosal1970} 
(see \cite{Niki2002cpc,Ning2017-ars}) 
states that every graph $G$ with
$m$ edges and $\lambda (G) > \sqrt{m}$ 
has a triangle.
In 2023, Ning and Zhai \cite{NZ2021}
proved further that if
$  \lambda (G)\ge \sqrt{m}$,
then $G$ contains at least $\lfloor \frac{\sqrt{m}-1}{2} \rfloor$  triangles, unless $G$ is a complete bipartite graph. 
This result is tight as evidenced by adding an edge to the larger vertex part of $K_{a,b}$, where $a=2(\sqrt{m} +1)$ and $b=\frac{\sqrt{m}-1}{2}$. 
Inspired by this result, 
we conjecture that the number of triangles is close 
to be doubling under the same constraint $\tau_3(G)\ge 2$ as in Theorem \ref{thm-n-2}. 

\begin{conjecture} \label{conj-71}
If $G$ is an $m$-edge graph with $\lambda (G)> 
 \sqrt{m}$, and there is no vertex contained in all triangles of $G$, 
 then $G$ contains at least $\sqrt{m} - O(1)$ triangles\footnote{We write $f(m)=O(1)$ if there are   positive constants $C$ and $m_0$ such that 
 $f(m)\le C$ for all $m\ge m_0$.}. 
\end{conjecture}

\noindent 
{\bf Remark.} 
The bound in Conjecture \ref{conj-71} 
is asymptotically tight. 
Indeed, for each $b\in \mathbb{N}$, we take 
$a=4b+3$ and $m=ab+2$, then we define $K_{a,b}^{+2}$ as the graph obtained from 
a complete bipartite graph $K_{a,b}$ by embedding  
two disjoint edges into the vertex part of size $a$. 
We can check that $K_{a,b}^{+2}$ has spectral radius 
larger than $\sqrt{m}$, and it contains $2b=\frac{1}{4}
(\sqrt{16m-23} -3) \approx \sqrt{m}$ triangles. 
This example shows the tightness of Conjecture \ref{conj-71}.

\medskip 
Theorem \ref{thm-bowtie}  gives a spectral counting result for the bowtie.  
Recall in Theorem \ref{thm-KMP} that 
 for sufficiently large $n$ and $q=o(n^2)$,  the only way to construct a graph 
with $\mathrm{ex}(n,F_2) + q$ edges and as few bowties as possible 
is to add some edges to $T_{n,2}$.   
 Our result could be regarded as a spectral version  of Theorem \ref{thm-KMP} in the case $q=2$. 
It is meaningful to prove a spectral version for a general $q\ge 3$. 
The first barrier we may encounter is to establish 
a spectral condition corresponding to the 
edge condition $e(G)\ge \lfloor n^2/4 \rfloor+ q$. 
Now, we denote by $Y_{n,2,q}$ the graph obtained from 
$T_{n,2}$ by embedding $q$ pairwise disjoint edges into the vertex part of size $\lceil n/2\rceil$. 
We propose the following conjecture.  

\begin{conjecture}
If $q\ge 3$ and $G$ is a graph with large order $n$ and  
\[  \lambda (G) \ge \lambda (Y_{n,2,q}), \]
then $G$ has at least ${q\choose 2} \lfloor \frac{n}{2}\rfloor$ 
bowties, and $Y_{n,2,q}$ is the unique spectral extremal graph. 
\end{conjecture} 

\noindent 
{\bf Remark.}  
We have a premonition that,  
with carefully structural analysis of the extremal graph, the supersaturation-stability method employed in this article could potentially be applicable to investigating this counting problem,  
particularly for a small integer $q$.  

\medskip 
We end this paper with a problem involving books.  
The booksize $bk(G)$ of a graph $G$ is defined as the largest number of triangles 
that share an edge.  An old conjecture of Erd\H{o}s, 
confirmed by Edwards (unpublished)
and independently by 
Khad\v{z}iivanov and Nikiforov (unavailable), 
asserts that every graph $G$ of order $n$ with more than $n^2/4$ edges 
has $bk(G)> n/6$. 
 We refer the readers to \cite[Lemma 4]{EFR1992} for the history 
 and  \cite{BN2005,LFP2024-triangular} for alternative proofs.   
In 2023, 
Zhai and Lin \cite{ZL2022jgt} proved that whenever $\lambda (G)\ge \lambda (T_{n,2})$, we have $bk(G)> 2n/13$, unless $G=T_{n,2}$. Furthermore, they proposed the following problem. 

\begin{problem}[Zhai--Lin, 2023] \label{prob-ZL}
For an arbitrary positive integer $n$, if $G$ is a graph of order $n$ with $\lambda (G)> \lambda (T_{n,2})$, 
is it true that $bk(G)> {n}/{6}$?  
\end{problem}

\noindent 
{\bf Remark.} 
Very recently, the present authors \cite{LFP2024-triangular} provided a short  proof of the Erd\H{o}s conjecture 
by using Theorem \ref{thm-far-bipartite}. 
In the concluding remarks of \cite{LFP2024-triangular}, 
we pointed out that a possible way to solve Problem \ref{prob-ZL} is to establish 
a spectral version of Theorem \ref{thm-far-bipartite}. 

\medskip 
In the other direction, Nikiforov \cite{Niki2021} 
showed that an $m$-edge graph $G$ with $\lambda (G)\ge \sqrt{m}$ 
has $bk(G)> \frac{1}{12}\sqrt[4]{m}$, unless $G$ is a  complete bipartite graph. 
Recently, Ning and Zhai \cite{NZ2021b} 
studied the spectral supersaturation for $4$-cycles 
by showing that $\# C_4 \ge \frac{1}{2000}m^2$ is provided by $\lambda (G)> \sqrt{m}$. 
Furthermore, they proposed a problem on the precise counting of $4$-cycles.   
In a forthcoming paper, the first author together with Hong Liu and Shengtong Zhang will improve Nikiforov's result and 
resolve Ning--Zhai's problem. Moreover, we 
show some related spectral results 
on counting other substructures. $\heartsuit$

\section*{Acknowledgements} 
The authors  are grateful to Xiaocong He and Loujun Yu for proofreading the early version of this paper. The authors also would like to express their sincere gratitude to the anonymous referees for their careful reading and invaluable suggestions. 
This paper is equally contributed. 
Yongtao Li was supported by the Postdoctoral Fellowship Program of CPSF (No. GZC20233196), 
 Lihua Feng was  the NSFC grant (Nos. 12271527 and 12471022), 
 and Yuejian Peng was supported by the NSF of Hunan Province (No. 2025JJ30003) and 
the NSFC grant (Nos. 11931002 and 12371327).  
This work was also partially supported by the NSF of Hunan Province  (2023JJ30180) and  the NSFC grant (No. 12201202).

\end{document}